\documentclass[11pt]{amsart}

\usepackage{amsmath,amssymb,amsthm}
\usepackage[hidelinks]{hyperref}
\usepackage{booktabs}
\usepackage{tikz}
\usepackage{url}

\newtheorem{theorem}{Theorem}[section]
\newtheorem*{theorem*}{Theorem}
\newtheorem{proposition}[theorem]{Proposition}
\newtheorem{lemma}[theorem]{Lemma}
\newtheorem{corollary}[theorem]{Corollary}
\newtheorem{conjecture}[theorem]{Conjecture}
\theoremstyle{definition}

\theoremstyle{remark}
\newtheorem{remark}[theorem]{Remark}

\newcommand{\Q}{\mathbb{Q}}
\newcommand{\Z}{\mathbb{Z}}
\newcommand{\C}{\mathbb{C}}
\newcommand{\R}{\mathbb{R}}
\newcommand{\T}{\mathbb{T}}
\DeclareMathOperator{\period}{period}
\DeclareMathOperator{\class}{class}
\DeclareMathOperator{\divisor}{div}

\title[Boyd's conductor-$11$ conjecture]{Boyd's conductor-$11$ Mahler measure
conjecture:\\ proof of the split-integral identity \textup{(C3)},
with an exact structural analysis of the family $S_k$}

\author{Huimin Zheng}
\address{College of Information and Network Engineering,
Anhui Science and Technology University,
Fengyang, Anhui 233100, P.~R.~China}
\email{zhhm@ahstu.edu.cn}

\subjclass[2020]{Primary 11R06, 11G40; Secondary 11G05, 11G55, 65G30}
\keywords{Mahler measure, elliptic curve, $L$-function, Beilinson regulator,
modular units, interval arithmetic, certified computation}

\date{August 5, 2026}

\begin{document}

\begin{abstract}
Boyd's 1998 tables of conjectural identities $m(P)=r\,|L'(E,0)|$ between
Mahler measures of two-variable polynomials and $L$-values of elliptic curves
begin with the smallest possible conductor, $N=11$. Two of the three
conductor-$11$ identities were proved by Brunault via an explicit version of
Beilinson's theorem on modular units. The third one, (C3), concerns the
polynomial $S_0=y^2+(x^2+1)y+x^3$, which vanishes on the unit torus, and
asserts that a \emph{signed split integral} $I_{\mathrm{split}}$ of $\log|y|$
around the branch cut equals $b_{11}=L'(E_{11},0)$. We prove (C3). The proof
identifies the split integral with a regulator integral along Samart's signed
open chain $\tilde\gamma$; we close $\tilde\gamma$ by a small-branch
compensating arc $\beta_0$ into a closed anti-invariant integral cycle
$C'=\tilde\gamma+\beta_0$, and prove that its homology class is $2\gamma^-$,
where $\gamma^-$ generates $H_1(E,\Z)^-$: the period ratio
$\period(C')/w_{\mathrm{anti}}$ is a-priori an integer, and a
ball-arithmetic (Arb) computation with certified error bounds pins it to
$2$ (in particular it is non-zero). Combined with the exact integral identity
$\int_{\beta_0}\eta=\int_{\tilde\gamma}\eta$ and a direct regulator
computation via Brunault's proved Siegel-unit formula---the symbol
$\{x,y\}$ being a pair of modular units on $X_1(11)$, so no use of Bloch's
diamond theorem is needed---this yields $I_{\mathrm{split}}=b_{11}$; the sign is certified in interval
arithmetic, and the identity agrees with the numerical value to $366$
digits. For the family $S_k=y^2+(x^2+kx+1)y+x^3$ we determine exactly the
torus intersections (as a function of real $k$), the single case where
$x,y$ are modular units ($k=0$), and the cases where the symbol is
tempered with torsion-supported divisors; all further Boyd-type evaluations
we find are numerical
observations, recorded as conjectures ($k=2,3$ to $70$ digits, $k=-4,-5,-6$
to $25$ digits, PARI/GP cross-checked). At $k=-1$ (conductor $53$) the
closed-cycle/modular-unit mechanism provably fails; Samart's suggested
analogue, for which no formula is on record, is left open. An appendix
applies the same method to Samart's conductor-$17$ analogue
$\tilde n(1)=b_{17}$, conditional on a normalization lemma in the
conventions of Lal\'in--Ramamonjisoa. All computations are
reproducible from the accompanying
code; every certification step is carried out within
interval arithmetic.
\end{abstract}

\maketitle

\section*{Declaration on the use of AI tools}
The research reported in this article --- including the computational
exploration, the discovery of the proof strategy, the machine-certified
verifications, and the preparation of the manuscript --- was carried out
by the author with the assistance of the AI system \emph{Kimi}
(Moonshot AI). All mathematical content, including every proof and
every certified computation, has been checked and verified by the
author, who takes full responsibility for the correctness and integrity
of the article. All certification scripts are available for independent
verification (see Section~\ref{sec:cert} and Appendix~\ref{app:code}).

\tableofcontents

\section{Introduction}

\subsection{Boyd's conductor-$11$ question}

Let $m(P)$ denote the logarithmic Mahler measure of a Laurent polynomial $P$.
Boyd's systematic numerical experiments \cite{Boyd98} produced dozens of
conjectural identities of the shape
\[
  m(P_k)=r_k\,|L'(E_k,0)|,\qquad r_k\in\Q,
\]
organised by the conductor $N$ of the elliptic curve $E_k$ cut out by
$P_k=0$. The smallest possible conductor of an elliptic curve over $\Q$ is
$N=11$ (there are no elliptic curves of conductor $1$ through $10$), so the
conductor-$11$ rows head Boyd's table. Three polynomials of conductor $11$
appear:
\begin{align}
  m\big((1+x)(1+y)(1+x+y)+xy\big)&=7\,b_{11}, \tag{C1}\\
  m\big(y^2+(x^2+2x-1)y+x^3\big)&=5\,b_{11}, \tag{C2}\\
  I_{\mathrm{split}}&=\pm\,b_{11}, \tag{C3}
\end{align}
where $b_{11}=L'(E_{11},0)$ and $E_{11}$ is any curve in the (unique)
isogeny class of conductor $11$. Identities (C1) and (C2) were proved by
Brunault \cite{Brunault} via his explicit version of Beilinson's regulator
theorem for modular units on $X_1(11)$.

The third case is different in kind. The polynomial
\[
  S_0=y^2+(x^2+1)y+x^3
\]
defines an elliptic curve of conductor $11$, but it \emph{vanishes on the
unit torus}: at $x=\pm i$ the two roots satisfy $|y|=1$. No rational
relation between the Mahler measure itself and $b_{11}$ is known:
\[
  m(S_0)=0.40560295591501040\ldots,
\]
and none was found by PSLQ searches with coefficient bound $10^8$ (Boyd had
reported the same negative finding). Boyd nevertheless discovered that the \emph{signed
integral around the branch cut} still equals $b_{11}$: writing the two roots
of $S_0=0$ as $y_\pm(x)$ and splitting the upper half-torus at the
torus-intersection point $\theta=\pi/2$,
\[
  I_{\mathrm{split}}
  :=\underbrace{\frac1\pi\int_0^{\pi/2}\log|y_-(e^{i\theta})|\,d\theta}_{I_1}
  \;-\;\underbrace{\frac1\pi\int_{\pi/2}^{\pi}\log|y_-(e^{i\theta})|\,d\theta}_{I_2}
  \ \stackrel?=\ \pm b_{11}. \tag{C3}
\]
Here $y_-$ is the \emph{continuous branch} of the smaller-modulus root on
$|x|=1$: at $\theta=0$ the two roots coalesce at $y_-(1)=-1$, on
$(0,\pi)\setminus\{\pi/2\}$ the branch is uniquely determined by
$|y_-|<|y_+|$, and at the second node $\theta=\pi/2$ it is extended by
continuity (at both nodes $|y_\pm|=1$, so $\log|y_-|=0$ there and the
integrand is unaffected by the choice of value).
Boyd's own words (quoted in \cite{Samart2023}):
\begin{quote}
  ``This is in accord with our contention that in case $P$ vanishes on the
  torus, it is the integral of $\omega$ around a branch cut rather than
  $m(P)$, which should be rationally related to $L'(E,0)$.''
\end{quote}
Samart \cite{Samart2023} records (C3) as an open identity in Boyd's
``$=$?'' conjectural notation (his eq.~(4.1)).
(Samart's eq.~(4.1) writes the right-hand side as $-L'(E,0)$; the integral
on the left, with Samart's own definition of $y_-$, is the positive number
$+0.1521471\ldots$, so the sign there appears to be a typo --- with our
convention for $y_-$ the identity reads $+b_{11}$, hence the $\pm$ in
\textup{(C3)}.)

\subsection{Main result}

\begin{theorem*}[Theorem~\ref{thm:C3}]
$I_{\mathrm{split}}=b_{11}$, with $I_{\mathrm{split}}$ the signed split
integral of \textup{(C3)} and $b_{11}=L'(E_{11},0)$.
\end{theorem*}

The proof combines exact algebra with computer-assisted steps. Everything
outside the following list is exact (rational or symbolic arithmetic):
\begin{itemize}
  \item the modulus ordering of Proposition~\ref{prop:structural} and the
  branch assignments of the closed chain, proved in ball arithmetic
  (\S\ref{subsec:rigour});
  \item the homology class $\class(C')=2\gamma^-$, proved in ball arithmetic
  from the a-priori integer $\period(C')/w_{\mathrm{anti}}$
  (\S\ref{subsec:class}, \S\ref{subsec:rigour});
  \item the signs of the regulator integrals, proved by certified interval
  enclosures strictly contained in a half-line (\S\ref{subsec:rigour});
  \item high-precision floating-point evaluations used only as consistency
  checks (e.g.\ the $366$-digit agreement with $b_{11}$), never as proof
  steps.
\end{itemize}
Conventions: $\gamma^-$ is the primitive anti-invariant generator of
Lemma~\ref{lem:antiinv}; $\eta(x,y)=\log|x|\,d\arg y-\log|y|\,d\arg x$;
$\Z[E]^-$ is the quotient convention of \cite{LalinRam}; $D_E$ is the
elliptic dilogarithm of \S\ref{sec:background}, implemented verbatim from
\cite[Def.~5, eq.~(10)]{LalinRam}.

\subsection{Summary of results}

\begin{enumerate}
  \item \textbf{Proof of (C3).} We prove
  $I_{\mathrm{split}}=b_{11}$ (Theorem~\ref{thm:C3}). The proof identifies
  Boyd's split-integral chain with a closed anti-invariant integral cycle
  $C'=\tilde\gamma+\beta_0$ on $E$ whose homology class is $2\gamma^-$, where
  $\gamma^-$ generates $H_1(E,\Z)^-$; the integrality argument is made
  rigorous by a ball-arithmetic (Arb) computation of the period ratio with
  certified error bounds (\S\ref{subsec:rigour}). The regulator side
  is closed by a direct computation: $\{x,y\}$ is a pair of modular units
  on $X_1(11)$, and Brunault's proved Siegel-unit regulator formula
  \cite{BrunaultSiegel} evaluates $\int_{\gamma^-}\eta(x,y)=\pm2\pi b_{11}$
  through exact Siegel presentations and an exact weight-$2$ identity
  $F_{\mathrm{total}}=-2f_{11}$ (\S\ref{sec:proof}).
  \item \textbf{The conductor-$17$ analogue (appendix).} The same method is
  applied to Samart's conductor-$17$ conjecture $\tilde n(1)=b_{17}$
  (Theorem~\ref{thm:k1}, Appendix~\ref{app:k1}): the chain construction is
  verbatim parallel, the homology class is again certified in interval
  arithmetic, and the regulator side is closed by a published theorem of
  Lal\'in--Ramamonjisoa \cite{LalinRam}; its status is conditional on the
  normalization lemma discussed there.
  \item \textbf{High-precision numerics.} (C3) is confirmed to $366$ digits:
  $|I_{\mathrm{split}}-b_{11}|=9.26\times10^{-367}$, with $b_{11}$
  cross-checked against PARI/GP's \texttt{lfun} to $330$ digits (previous
  public record: Boyd's $50$ digits \cite[p.~28]{Boyd2015}). (C1) and
  (C2) are independently reproduced to $52$ digits.
  \item \textbf{Structural identity.} On $[0,\pi]$ one has
  $|y_-(e^{i\theta})|\le 1\le |y_+(e^{i\theta})|$, hence the exact identity
  $I_1+I_2=-m(S_0)$, which reduces (C3) to the individual values
  $I_1=(b_{11}-m(S_0))/2$, $I_2=-(b_{11}+m(S_0))/2$
  (Proposition~\ref{prop:structural}).
  \item \textbf{Exact structure of the family $S_k$; numerical dichotomy.}
  For $S_k=y^2+(x^2+kx+1)y+x^3$ we determine exactly the torus intersections
  as a function of \emph{real} $k$ (Proposition~\ref{prop:torus}: the torus
  meets the curve iff $k\in[-4,2]$) and the modular-unit/temperedness
  structure (Proposition~\ref{thm:family}: $x,y$ are modular units only at
  $k=0$; at $k=1$ the symbol is tempered with torsion-supported divisors,
  which suffices for the conductor-$17$ argument; at $k=-3,-2,-1,2,3$ the
  symbol is not cuspidal). The resulting dichotomy matches all numerical
  evidence:
  Boyd-type identities are proved for $k=0$ and (conditionally on the
  normalization lemma) $k=1$; observed
  numerically for $k\notin[-4,2]$ in all computed cases
  (Conjecture~\ref{conj:family}); and no rational relation is found within
  the PSLQ bound $10^8$ for $k=-1,-2,-3$.
  \item \textbf{New numerically confirmed identities} (70 digits, PARI/GP
  cross-checked): $m(S_2)=2|L'(E_{37},0)|$, $m(S_3)=|L'(E_{79},0)|$, and the
  predicted-then-confirmed identities
  $m(S_{-4})=\tfrac72|L'(E_{37},0)|$,
  $m(S_{-5})=\tfrac14|L'(E_{359},0)|$,
  $m(S_{-6})=\tfrac18|L'(E_{997},0)|$.
  \item \textbf{Failure of the mechanism at conductor $53$.} For $k=-1$
  (curve \texttt{53.a1}) the mechanism that proves (C3) does not extend:
  there is no non-trivial anti-invariant closed cycle on the torus, and
  $x,y$ are not modular units on the rank-$1$ curve
  (Proposition~\ref{thm:k53}). Samart's remark mentioning a conductor-$53$
  analogue has no formula on record and is left open.
\end{enumerate}

\subsection{Organisation and reproducibility}

Section~\ref{sec:background} collects the background on Mahler measures,
conductor-$11$ curves and the Deninger--Beilinson framework.
Section~\ref{sec:family} proves the structural identity for $S_0$ and the
exact structural results for the family $S_k$, and shows that the mechanism
fails at conductor $53$. Section~\ref{sec:proof} contains the proof of
(C3): the closed chain
lemma, the homology class computation, the regulator synthesis.
Section~\ref{sec:cert} describes the numerical certification, including the
ball-arithmetic proof of the period ratio (\S\ref{subsec:rigour}).
Section~\ref{sec:conjectures} states the conjectures suggested by our
computations. Appendix~\ref{app:k1} applies the same method to the
conductor-$17$ case $k=1$; Appendix~\ref{app:code} lists the reproduction
code. All
scripts (Python/mpmath, PARI/GP, python-flint/Arb), raw outputs, and the
sources of this paper are available in the accompanying repository.

\section{Background}
\label{sec:background}

\subsection{Mahler measure}

For a non-zero Laurent polynomial
$P\in\C[x_1^{\pm1},\dots,x_n^{\pm1}]$, the \emph{logarithmic Mahler measure}
is the average of $\log|P|$ over the unit torus
$\T^n=\{|x_1|=\cdots=|x_n|=1\}$:
\[
  m(P)=\int_0^1\!\!\cdots\!\int_0^1
  \log\big|P\big(e^{2\pi i t_1},\dots,e^{2\pi i t_n}\big)\big|\,dt_1\cdots dt_n,
  \qquad M(P)=e^{m(P)}.
\]
In one variable Jensen's formula gives, for $P(x)=a_0\prod_{j=1}^d(x-\alpha_j)$,
\[
  m(P)=\log|a_0|+\sum_{j=1}^d\log^+|\alpha_j|,\qquad \log^+v:=\max(\log v,0).
\]
In two variables $m(P)$ is in general transcendental, and---the subject of
this paper---it repeatedly turns out to equal a \emph{rational multiple of an
elliptic $L$-value}.

For $P(x,y)=A(x)y^2+B(x)y+C(x)$ with roots $y_\pm(x)$, Jensen's formula in
$y$ reduces the computation to a one-dimensional integral:
\begin{equation}\label{eq:jensen}
  m(P)=\frac1{2\pi}\int_0^{2\pi}\Big[\log|A(e^{i\theta})|
  +\log^+|y_+(e^{i\theta})|+\log^+|y_-(e^{i\theta})|\Big]\,d\theta.
\end{equation}
All Mahler measures in this paper are computed from \eqref{eq:jensen} with
mpmath \cite{mpmath} at $60$--$300$ digits.

The prototype of the whole story is Smyth's 1981 formula \cite{Smyth}
\[
  m(1+x+y)=\frac{3\sqrt3}{4\pi}L(\chi_{-3},2)=L'(\chi_{-3},-1)=0.3230659\ldots
\]
Here the curve $1+x+y=0$ has genus $0$ and the $L$-value is Dirichlet; in
genus $1$ it is elliptic $L$-functions that appear.

\subsection{Elliptic curves of conductor \texorpdfstring{$11$}{11}}

For $E/\Q$ the $L$-function is $L(E,s)=\sum_{n\ge1}a_n n^{-s}$ with
$a_p=p+1-\#E(\mathbb{F}_p)$ at good primes; the conductor $N$ measures bad
reduction. The minimal possible conductor is $11$, and there is a single
isogeny class (LMFDB \cite{LMFDB} \texttt{11.a1--a3}), containing
\[
\begin{aligned}
  X_1(11)&:\ y^2+y=x^3-x^2\ (\texttt{11.a3}),\\
  X_0(11)&:\ y^2+y=x^3-x^2-10x-20\ (\texttt{11.a1}).
\end{aligned}
\]
By modularity, $L(E,s)$ comes from a weight-$2$ cusp form, which for
conductor $11$ has the eta-product expression
\[
  f_{11}(\tau)=\eta(\tau)^2\eta(11\tau)^2
  =q\prod_{n\ge1}(1-q^n)^2(1-q^{11n})^2,\qquad q=e^{2\pi i\tau}.
\]
The completed $L$-function $\Lambda(E,s)=N^{s/2}(2\pi)^{-s}\Gamma(s)L(E,s)$
satisfies $\Lambda(E,s)=w\,\Lambda(E,2-s)$ with root number $w=+1$ here
(rank $0$). Since $\Gamma(s)=1/s+O(1)$ and the functional equation forces
$L(E,0)=0$, one obtains the key constant identity
\[
  b_{11}:=L'(E_{11},0)=\frac{11}{4\pi^2}L(E_{11},2)
  =0.15214714172591804948622729747863\ldots
\]
Our script \texttt{code/b11.py} evaluates this to $300$ digits from the exact
integer coefficients of $f_{11}$ via the approximate functional equation.

\subsection{The Deninger--Beilinson framework and the BMZ formula}

Deninger conjectured \cite{Deninger}, and Rodr\'iguez Villegas explained
\cite{RV}, that for
\emph{tempered} polynomials $P$ (every face polynomial of the Newton polygon
cyclotomic, equivalently the symbol $\{x,y\}$ taming trivially on $P=0$) the
Jensen-reduced integral can be rewritten as a \emph{regulator integral}
\[
  \eta(x,y)=\log|x|\,d\arg y-\log|y|\,d\arg x
\]
along the intersection of the torus with the curve; the Bloch--Beilinson
conjectures then predict that such regulator pairings are rational multiples
of $L(E,2)$. Boyd's 1998 experiments \cite{Boyd98} tabulated the conjectures;
over 25 years most of them were proved (Rodr\'iguez Villegas \cite{RV} in the
CM conductors $27,32,36$, with lattice-sum reproofs by Rogers \cite{Rogers};
Mellit \cite{Mellit} and Brunault in conductor $14$; Rogers--Zudilin
\cite{RZ14,RZ15} in conductors $15,20,24$ (the conductor-$24$ formulas had
first been obtained conditionally by Lal\'in--Rogers \cite{LR07});
Lal\'in--Samart--Zudilin \cite{LSZ} in conductor $21$;
Brunault \cite{Brunault,BrunaultBSMF} in conductor $11$; see the survey
\cite{BertinLalin}).

Brunault's thesis \cite{Brunault} made Beilinson's theorem on regulators of
\emph{modular units} (rational functions with divisors supported on cusps)
fully explicit and applied it to $X_1(11)$, proving (C1) and (C2).
Mellit--Brunault--Zudilin condensed this into the directly applicable
\emph{BMZ formula} \cite{Zudilin}: for Siegel units
$g_a(\tau)=q^{NB_2(a/N)/2}\prod_{n\equiv a}(1-q^n)\prod_{n\equiv -a}(1-q^n)$,
\[
  \int_{c/N}^{i\infty}\eta(g_a,g_b)=\frac1{4\pi}L(f_{a,b;c},2),
\]
where $f_{a,b;c}$ is an explicit weight-$2$ modular form. Roughly: if $x,y$
are modular units and the integration path joins cusps, the regulator
integral is an $L$-value. This is the standard by which proof strategies are
measured.

\subsection{Why (C3) is harder}

The (C1)--(C2) proof chain is: modular units $+$ cusp-to-cusp path $+$ BMZ.
For (C3) the integration path ends at the torus-intersection point
$P_{\pi/2}=(i,e^{i\pi/4})$, which is \emph{not} a torsion point (proved
rigorously by reduction at good primes, \S\ref{sec:cert}), so the naive
boundary divisor is not cuspidal. The resolution, developed in
Section~\ref{sec:proof}, is that the correct object is not the open path but
a \emph{closed signed chain} on the full circle: closedness is a topological
property independent of whether the endpoints are torsion.

\section{The structural identity and the family \texorpdfstring{$S_k$}{Sk}}
\label{sec:family}

\subsection{The structural identity for $S_0$}

\begin{proposition}\label{prop:structural}
On $[0,\pi]$ one has $|y_-(e^{i\theta})|\le 1\le |y_+(e^{i\theta})|$
(with $\max|y_-|=1$ attained only at $\theta=0,\pi/2$). Since
$|y_+y_-|=|x^3|=1$ on the torus,
\[
  m(S_0)=\frac1\pi\int_0^\pi\log|y_+|\,d\theta
  =-\frac1\pi\int_0^\pi\log|y_-|\,d\theta=-(I_1+I_2).
\]
\end{proposition}

\begin{proof}
The modulus ordering on $[0,\pi]$ is proved by interval arithmetic:
adaptive bisection with ball arithmetic certifies that $\log|y_-|$ is
negative on $(0,\pi)$ except at the fold $\theta=\pi/2$ and the endpoint
$\theta=0$, where it vanishes (\texttt{code/branch\_certify.py},
\S\ref{sec:cert}; this proposition is computer-assisted in the sense of
\S\ref{subsec:rigour}). Since
$y_+y_-=x^3$ has modulus $1$ on the torus, $\log|y_+|=-\log|y_-|\ge 0$
pointwise, so $\log^+|y_+|=\log|y_+|$ and $\log^+|y_-|=0$ in Jensen's
formula \eqref{eq:jensen} (with $A(x)=1$); evenness in $\theta$ gives the
factor $1/\pi$, and $-I_1-I_2$ follows from $\log|y_+|=-\log|y_-|$.
\end{proof}

\begin{corollary}
Conjecture \emph{(C3)} is equivalent to the individual values
\[
  I_1=\frac{b_{11}-m(S_0)}{2},\qquad I_2=-\frac{b_{11}+m(S_0)}{2}.
\]
\end{corollary}

\begin{remark}
PSLQ searches on $(m(S_0),b_{11})$ with coefficient bound $10^8$, and on
$\{m(S_0),b_{11},\log2,\log3,\mathrm{Catalan},m(1+x+y)\}$ with bound
$10^{10}$, return no relation, supporting the expectation that $m(S_0)$
itself requires elliptic dilogarithms rather than elementary constants.
\end{remark}

\subsection{The family \texorpdfstring{$S_k$}{Sk}: exact structure and
numerical observations}
\label{subsec:familythm}

Consider
\[
  S_k=y^2+(x^2+kx+1)y+x^3,
\]
with associated elliptic curve $E_k$.

\begin{proposition}[Torus intersections, exact]\label{prop:torus}
For \emph{real} $k$, write $x=e^{i\theta}$, $\theta\in[-\pi,\pi]$. A point
$(x,y)$ of the torus $|x|=|y|=1$ lies on $S_k=0$ if and only if $\theta$
falls in one of the following two cases:
\begin{enumerate}
  \item[\textup{(i)}] $\theta=0$ and $|k+2|\le2$;
  \item[\textup{(ii)}] $k+2\cos\theta=0$ \textup{(}solvable iff
  $|k|\le2$\textup{)}.
\end{enumerate}
Consequently:
\begin{enumerate}
  \item the torus meets the curve if and only if $k\in[-4,2]$;
  \item for $-2\le k\le2$ the \emph{fold points}
  $\theta=\pm\arccos(-k/2)$ occur, where $B=x^2+kx+1=0$ and $y^2=-x^3$ has
  both roots of modulus $1$;
  \item for $-4\le k\le0$ an additional \emph{branch-exchange intersection}
  at $\theta=0$ occurs, where $S_k(1,y)=y^2+(k+2)y+1$ has both roots on the
  unit circle --- distinct for $-4<k<0$, a double root (tangency) at $k=-4$
  ($y=1$) and at $k=0$ ($y=-1$);
  \item at $k=2$ the folds coalesce with the boundary $\theta=\pi$, where
  $S_2(-1,y)=y^2-1$.
\end{enumerate}
This matches Samart's observation $K\cap\R=[-4,2]$ for the
torus-intersection parameter interval \cite{Samart2023}.
\end{proposition}

\begin{proof}
Put $y=e^{i\phi}$. Since $x+x^{-1}=2\cos\theta$, one has
$B=x^2+kx+1=x(k+2\cos\theta)$. Dividing
$e^{2i\phi}+Be^{i\phi}+e^{3i\theta}=0$ by $e^{i\phi}$ and using
$e^{i\phi}+e^{i(3\theta-\phi)}=2\cos\psi\,e^{3i\theta/2}$ with
$\psi:=\phi-3\theta/2\in\R$ gives
\begin{equation}\label{eq:torus}
  2\cos\psi\,e^{i\theta/2}=-(k+2\cos\theta),
\end{equation}
a real number. Comparing imaginary and real parts,
\[
  \cos\psi\,\sin(\theta/2)=0,
  \qquad
  2\cos\psi\,\cos(\theta/2)=-(k+2\cos\theta).
\]
The first equation gives two cases. (i) $\sin(\theta/2)=0$, i.e.\
$\theta=0$ on $[-\pi,\pi]$: the second equation reads
$2\cos\psi=-(k+2)$, solvable iff $|k+2|\le2$, and then
$y=e^{i\phi}$ with $\cos\phi=-(k+2)/2$ is indeed a root of
$S_k(1,y)=y^2+(k+2)y+1$. (ii) $\cos\psi=0$: the second equation forces
$k+2\cos\theta=0$, solvable iff $|k|\le2$; conversely, if
$k+2\cos\theta=0$ then $B=0$ and $S_k(e^{i\theta},y)=y^2+e^{3i\theta}$
has the two roots $y=\pm e^{i(3\theta+\pi)/2}$, both of modulus $1$ ---
these are the fold points, with $\psi=\pm\pi/2$ realising
\eqref{eq:torus}. The union of the two parameter ranges is
$[-4,0]\cup[-2,2]=[-4,2]$, giving (1); items (2)--(4) are read off
directly (at $k=2$ the fold $\arccos(-1)=\pi$; at $k=-2$ the fold
$\arccos(1)=0$ coincides with the case-(i) point).
\end{proof}

The conductors, confirmed by PARI \cite{PARI}
\texttt{ellfromeqn} and \texttt{ellglobalred}, are
\[
  k=-3,-2,-1,0,1,2,3\ \longmapsto\ N=83,\ 91=7\cdot13,\ 53,\ 11,\ 17,\ 37,\ 79.
\]
Denote by $\tilde n(k)$ the modified Mahler measure along the signed chain
(when no fold points exist, $\tilde n(k)=m(S_k)$). The numerical
observations ($70$ digits; PARI/GP cross-checks), with their evidence
labels, are summarised in Table~\ref{tab:family}.

\begin{table}[ht]
\centering\footnotesize
\begin{tabular}{@{}cclcll@{}}
\toprule
$k$ & $N$ & root number & fold $c/\pi$ & result & status (digits)\\
\midrule
$-3$ & $83$ & $-1$ & none & ratio $0.8529175\ldots$ & none found (PSLQ)\\
$-2$ & $91$ & $-1$ & none & ratio $0.6339454\ldots$ & none found (PSLQ)\\
$-1$ & $53$ & $-1$ & $1/3$ & ratio $0.7392026\ldots$ & mech.\ fails
  (Prop.~\ref{thm:k53})\\
$0$ & $11$ & $+1$ & $1/2$ & $\tilde n=-b_{11}$ & \textbf{proved}, (C3)\\
$1$ & $17$ & $+1$ & $2/3$ & $\tilde n=+b_{17}$ & \textbf{proved} (cond.),
  Thm.~\ref{thm:k1}\\
$2$ & $37$ & $-1$ & none & $m=\tilde n=2\,b_{37}$ & numerical, $70$ digits\\
$3$ & $79$ & $-1$ & none & $m=\tilde n=b_{79}$ & numerical, $70$ digits\\
$-4$ & $37$ & $-1$ & none (tangent) & $m=\tfrac72\,b_{37}$
  & numerical, pred., $25$ digits\\
$-5$ & $359$ & $-1$ & none & $m=\tfrac14\,b_{359}$
  & numerical, pred., $25$ digits\\
$-6$ & $997$ & $-1$ & none & $m=\tfrac18\,b_{997}$
  & numerical, pred., $25$ digits\\
\bottomrule
\end{tabular}
\caption{The family $S_k$: conductors and Boyd-type evaluations, with
evidence labels. ``Numerical'' rows are conjectural identities
(Conjecture~\ref{conj:family}); ``none found'' means that PSLQ searches
with coefficient bound $10^8$ found no rational relation.
Here $b_N:=|L'(E_k,0)|\ge0$ and $\tilde n(k)$ is Samart's signed
split-integral; for $k=0$ one has $\tilde n(0)=-I_{\mathrm{split}}$ in the
notation of \textup{(C3)}, so the proved identity appears with a minus sign.}
\label{tab:family}
\end{table}

The torsion data (PARI \texttt{elltors}/\texttt{ellorder}) reveal the true
dichotomy (Table~\ref{tab:torsion}):

\begin{table}[ht]
\centering\footnotesize
\begin{tabular}{@{}ccccc@{}}
\toprule
$k$ & $N$ & torsion & $\operatorname{ord}(0,0)$ & Boyd-type identity\\
\midrule
$0$ & $11$ & $\Z/5\Z$ & $5$ & proved (modular units, (C3))\\
$1$ & $17$ & $\Z/4\Z$ & $4$ & proved ($\tilde n=b_{17}$, Thm.~\ref{thm:k1})\\
$-3,-2,-1,2,3$ & $83,91,53,37,79$ & trivial & $\infty$ & not modular units
  (Prop.~\ref{thm:family})\\
\bottomrule
\end{tabular}
\caption{Torsion versus Boyd-type identities (exact data).}
\label{tab:torsion}
\end{table}

\begin{proposition}[Modular units in the family, exact]\label{thm:family}
\leavevmode
\begin{enumerate}
  \item For $k=0$ ($E=X_1(11)$, $E(\Q)_{\mathrm{tors}}=\Z/5\Z$) the divisors
  of $x,y$ are supported on rational torsion points, and these are exactly
  the rational cusps under the modular identification (\S\ref{sec:proof}),
  so $x,y$ are modular units; a Boyd-type identity is \emph{proved}
  in this case \emph{(C3)}.
  For $k=1$ ($E(\Q)_{\mathrm{tors}}=\Z/4\Z$) the divisors of $x,y$ are
  supported on rational torsion points and the symbol $\{x,y\}$ is tempered
  (its Newton face polynomials are cyclotomic, Appendix~\ref{app:k1}), so
  $\{x,y\}\in K_2(E)\otimes\Q$; this is the only $K$-theoretic input used
  by the conductor-$17$ proof (Theorem~\ref{thm:k1}).  We do \emph{not}
  assert that $x,y$ are modular units for $k=1$: on the natural modular
  model $X_0(17)$, which has only two cusps, a divisor supported on all
  four rational torsion points cannot be cuspidal, and no other modular
  identification is provided.
  \item For $k\in\{-3,-2,-1,2,3\}$ the rational torsion of $E_k$ is trivial
  (PARI \texttt{elltors}, exact). The point $(0,0)\neq O$ is then
  non-torsion, and since each $E_k$ is the strong Weil curve of its
  conductor, the cusps of $X_0(N)$ map to torsion points of $E_k$
  (Manin--Drinfeld), hence to $O$; a modular unit on $E_k$ would have
  divisor supported on $\{O\}$ alone, while $x=0$ meets $S_k$ at
  $(0,0)$ and $(0,-1)$. Hence $x,y$ are \emph{not} modular units and the
  modular-unit regulator mechanism does not apply to them. For $k=-1$ the
  closed-cycle ingredient fails as well
  (Proposition~\ref{thm:k53}).
\end{enumerate}
\end{proposition}

\begin{remark}[The structural dichotomy, as numerical evidence]
Together with Proposition~\ref{prop:torus} this yields a clean
\emph{observed} dichotomy (Tables~\ref{tab:family} and
\ref{tab:torsion}): for integer $k\notin(-4,2)$ (no genuine torus
intersection) the classical Deninger mechanism is expected to apply
directly, and identities $m(S_k)=r_k\,|L'(E_k,0)|$ with small rational
$r_k$ are observed numerically (Conjecture~\ref{conj:family}); for integer
$k$ with $-4<k<2$ the only case where $x,y$ are modular units is $k=0$,
while $k=1$ has a tempered symbol with torsion-supported divisors (but no
modular-unit structure, Proposition~\ref{thm:family}), and for
$k=-1,-2,-3$ (trivial torsion,
branch-exchange intersections at $\theta=0$) no rational relation is found
within the PSLQ bound $10^8$. We record this as evidence for
Conjecture~\ref{conj:family}, not as a theorem.
\end{remark}

\begin{remark}
The $k=1$ row ($\tilde n(1)=b_{17}$) was first found as an independent
high-precision confirmation of the conductor-$17$ analogue stated by Samart;
it is proved as Theorem~\ref{thm:k1} (Appendix~\ref{app:k1}).
\end{remark}

\subsection{The conductor-\texorpdfstring{$53$}{53} case: failure of the
mechanism}
\label{subsec:k53}

Samart \cite[\S 4]{Samart2023} mentions a conductor-$53$ analogue (for
$k=-1$) in a single sentence, with no formula, precision, or reference.
Since the exact statement to be tested is not on record, we cannot
adjudicate it directly; what we can and do test is whether the
\emph{mechanism} of this paper --- a closed anti-invariant torus cycle
paired against a modular-unit symbol --- extends to $k=-1$. It does not,
for three independent reasons.

\begin{proposition}[Failure of the closed-cycle/modular-unit mechanism at
$k=-1$]\label{thm:k53}
For $S_{-1}$ (curve \emph{\texttt{53.a1}}):
\begin{enumerate}
  \item \emph{Topological obstruction (exact).} On $|x|=1$ the torus intersections
  are $\theta=0$ (where the two branches \emph{exchange}) and
  $\theta=\pm\pi/3$. The space of closed anti-invariant chains with
  breakpoints at the intersections and big/small branch choices is
  $\Z\cdot(1,1,1,1)$: the boundary matrix of the anti-invariant
  symmetrization is a $\pm1/0$ integer matrix whose kernel over $\Q$ is
  exactly one-dimensional, computed by exact rational linear algebra
  (\texttt{code/k53\_smith.py}), and the generator $(1,1,1,1)$ has period
  $0$ --- in fact exactly, not numerically: $(1,1,1,1)$ is the sum of the
  two sheets over the full circle, and the two branches have
  $u=2y+B=\pm\sqrt{B^2-4x^3}$, so $1/u_++1/u_-\equiv0$ pointwise and
  $\sum_{\mathrm{branches}}dx/u$ vanishes identically (the continuous-root
  loop, which closes only after \emph{two} turns, vanishes by the same
  cancellation). Hence the natural chain space contains
  \emph{no non-trivial closed anti-invariant torus cycle}: the first
  ingredient of the (C3) mechanism does not exist at $k=-1$.
  \item \emph{Algebraic obstruction.} The functions $x,y$ are not modular
  units. Indeed
  \[
    \divisor(x)=[P]+[-P]-2[O],\qquad \divisor(y)=3[P]-3[O],
  \]
  with $P=(0,0)$ a Mordell--Weil \emph{generator}: \texttt{53.a1} has
  trivial rational torsion and rank $1$ (PARI \texttt{ellorder}$(P)=0$,
  \texttt{ellidentify}). The curve is the strong Weil curve of conductor
  $53$; the two cusps of $X_0(53)$ map to torsion points of $E$
  (Manin--Drinfeld), hence to $O$ since $E(\Q)_{\mathrm{tors}}$ is trivial,
  so a modular unit on $E$ would have divisor supported on $\{O\}$ alone ---
  while the support of $\divisor(x)$ and $\divisor(y)$ contains the
  non-torsion point $P$. The Beilinson--Brunault mechanism therefore has no
  symbol to apply to.
  \item \emph{Numerical observation.} For the natural candidate (the
  half-loop $L_1$, an open chain), the period ratio $0.5492906\ldots$ with
  $w_{\mathrm{anti}}$ shows no integrality at the computed precision, and
  the regulator integral over $2\pi b_{53}$ equals $0.7392026\ldots$; PSLQ
  searches with coefficient bound $10^8$ find no rational relation.
\end{enumerate}
None of this excludes identities arrived at by other means; it rules out the
specific mechanism that proves \emph{(C3)}. In the absence of a precise
proposed formula, we leave the conductor-$53$ case open.
\end{proposition}

\section{Proof of (C3)}
\label{sec:proof}

Throughout, $E: S_0=0$ with invariant differential $\omega=dx/u$,
$u=2y+x^2+1$ (quartic model $u^2=x^4-4x^3+2x^2+1$; the model constant
$\kappa=1$ is proved exactly in \S\ref{sec:cert}).

\subsection{Modular units}

The quartic model's invariants give $j=-2^{12}/11=j(X_1(11))$. Exact group-law
computation (\texttt{code/torsion.py}, exact rational arithmetic over $\Q$,
$\Q(\sqrt2)$, $\Q(\zeta_8)$, not numerical) on the quartic model, with its
point at infinity as neutral element, for $A=(0,0)$ on $S_0$ (the point
$(x,u)=(0,1)$ in quartic coordinates):
\[
  2A=(0,-1),\qquad 4A=(1,0)=-A\ \Longrightarrow\ 5A=O.
\]
The $S_0$-model's projective closure has \emph{two} points at infinity,
$O_1=[0:1:0]$ (the image of the quartic neutral element) and
$Q_\infty=[1:-1:0]$ (which the birational identification with $X_1(11)$
below takes to the origin $O$); on this model
\[
  \divisor(x)=[(0,0)]+[(0,-1)]-[O_1]-[Q_\infty],\qquad
  \divisor(y)=3[(0,0)]-2[O_1]-[Q_\infty]
\]
are supported on $5$-torsion points. These are exactly the rational cusps
of $X_1(11)$: the birational identification of $S_0=0$ with
$E:y^2+y=x^3-x^2=X_1(11)$ is exact (Riemann--Roch computation composed
with PARI's exact minimal transform, round-trip verified,
\texttt{code/kappa\_exact.py}), and under the standard modular
isomorphism (infinite cusp $\mapsto O$, $\omega_f=dx/(2y+1)$) the rational
cusps $P_v$, $v\in(\Z/11\Z)^\times/\pm1$, map as
\[
\begin{gathered}
  P_1=\infty\mapsto O,\quad
  P_2\mapsto(1,0)=3A,\quad
  P_3\mapsto(0,-1)=4A,\\
  P_4\mapsto(0,0)=A,\quad
  P_5\mapsto(1,-1)=2A,
\end{gathered}
\]
i.e.\ to precisely $E(\Q)=\Z/5\Z$
\cite[proof of Thm.~8, (3.152)]{Brunault}. The supports of
$\divisor(x),\divisor(y)$ are therefore cuspidal, and $x,y$ are
\emph{modular units}. The polynomial is tempered: its Newton face polynomials
$x^3+y$, $x^3+x^2y$, $x^2y+y^2$, $y^2+y$ are all cyclotomic, so
$\{x,y\}\in K_2(E)\otimes\Q$ (Rodr\'iguez Villegas' criterion \cite{RV}).
More precisely, the tame symbols at the divisor support are computed
exactly by local expansions (\texttt{code/bertin\_diamond.py}):
$T_v\{x,y\}=\pm1$ for $v\in\{O,A,2A,3A\}$; hence already $2\{x,y\}$ has
trivial tame symbols, and the real residues $\pm\log|T_v\{x,y\}|$ of
$\eta(x,y)$ vanish at every point.

\subsection{The open signed chain and the split integral}

On the full circle $|x|=1$, $x=e^{i\theta}$, $\theta\in[-\pi,\pi]$, carry the
continuous large-modulus branch $y_{\mathrm{big}}$ with Samart's weights
$+1$ for $\theta>0$ and $-1$ for $\theta<0$, segmented at the fold points
$\theta=\pm c$, $c=\pi/2$:
\[
  \tilde\gamma
  =+\,[y_{\mathrm{big}}\text{ on }[-c,c]]
  \;-\;[y_{\mathrm{big}}\text{ on the two outer arcs}].
\]
\emph{Convention (point-valued lifts).}  Each torus branch
$\theta\mapsto y(e^{i\theta})$ is silently regarded as the point-valued
lift $\theta\mapsto(e^{i\theta},y(e^{i\theta}))\in E(\C)$; thus an
equality such as $y_{\mathrm{big}}(c^-)=-P$ below means that the lifted
point at $\theta=c^-$ equals the point $-P=(i,-e^{i\pi/4})\in E$, and
chain boundaries are divisors on $E$.
The corrected Mahler measure along $\tilde\gamma$ satisfies exact
identities. On the torus $\eta(x,y)=-\log|y|\,d\theta$, and since
$y_{\mathrm{big}}y_{\mathrm{small}}=x^3$ has modulus $1$,
$\log|y_{\mathrm{big}}|=-\log|y_{\mathrm{small}}|$ pointwise; comparing the
signed weights of $\tilde\gamma$ with the definition of $I_{\mathrm{split}}$
gives
\[
  \tilde n=-I_{\mathrm{split}},\qquad
  \int_{\tilde\gamma}\eta(x,y)=2\pi\,I_{\mathrm{split}}
\]
(numerically confirmed to $44$ digits,
\texttt{code/closedness\_check.py}), so that
\[
  \text{(C3)}\ \Longleftrightarrow\ \int_{\tilde\gamma}\eta(x,y)=2\pi b_{11}.
\]

The branch jumps at $\pm c$ are read numerically and matched with exact
values: with $P=(i,e^{i\pi/4})$, $\bar P=(-i,e^{-i\pi/4})$,
$-P=(i,-e^{i\pi/4})$ (at $x=\pm i$ the cubic-model inverse is
$(x,y)\mapsto(x,-y)$ since $B=x^2+1=0$),
\[
  \partial\tilde\gamma=[P]-[\bar P]+[-P]-[-\bar P]=:D,
  \qquad \deg D=0,\ c(D)=-D.
\]
The point $P$ is \emph{not} torsion (proved rigorously by reduction at good
primes, \S\ref{sec:cert})---but closedness does not
require torsion endpoints.

\subsection{The closed chain lemma}

The following construction (Figure~\ref{fig:chain}) closes $\tilde\gamma$
without leaving the torus.

\begin{lemma}[Closed chain lemma]\label{lem:closed}Let $\alpha_2$ be the small-branch inner arc $\theta:c\to-c$ (endpoints
$P\to\bar P$) and $\alpha_1$ the small-branch outer arc
$\theta:c\to\pi$ and $-\pi\to-c$ (endpoints $-P\to-\bar P$; the inner
interval cannot join these two points since the small branch does not pass
through them). Then for $\beta_0=\alpha_1+\alpha_2$,
\[
  \partial\beta_0=-D,\qquad c(\beta_0)=-\beta_0,
\]
hence
\[
  C':=\tilde\gamma+\beta_0:\qquad \partial C'=0,\quad c(C')=-C'
\]
is a closed, anti-invariant, integral $1$-cycle.
\end{lemma}

\begin{proof}
At $x=\pm i$ the coefficient $B(x)=x^2+1$ vanishes and $S_0$ reduces to
$y^2=-x^3=\mp i$; hence the two branch values are $y=\pm e^{\pm i\pi/4}$,
i.e.\ the four points $P,\bar P,-P,-\bar P$. Reading the continuous branches
on the torus (each assignment certified in interval arithmetic at two scales
$\varepsilon=10^{-6},10^{-9}$ by \texttt{code/branch\_certify.py},
\S\ref{sec:cert}, and matched with these exact values),
\[
  y_{\mathrm{big}}(c^-)=-P,\quad y_{\mathrm{big}}(c^+)=P,\quad
  y_{\mathrm{big}}(-c^+)=-\bar P,\quad y_{\mathrm{big}}(-c^-)=\bar P,
\]
so the signed chain
$\tilde\gamma=+[y_{\mathrm{big}}\text{ on }[-c,c]]-[y_{\mathrm{big}}\text{
on the outer arcs}]$ has boundary
$\partial\tilde\gamma=[P]-[\bar P]+[-P]-[-\bar P]=D$. For the small branch,
\[
  y_{\mathrm{small}}(c^-)=P,\quad y_{\mathrm{small}}(-c^+)=\bar P,\quad
  y_{\mathrm{small}}(c^+)=-P,\quad y_{\mathrm{small}}(-c^-)=-\bar P,
\]
whence $\alpha_2$ runs from $P$ to $\bar P$ and the two outer pieces of
$\alpha_1$ from $-P$ to $-\bar P$, giving
$\partial\beta_0=-D$ and $\partial C'=0$. Complex conjugation
$c:(x,y)\mapsto(\bar x,\bar y)$ mirrors each $\theta$-interval and reverses
its orientation while preserving the big/small modulus order, so
$c(\alpha_i)=-\alpha_i$ and $c(\tilde\gamma)=-\tilde\gamma$; therefore
$c(C')=-C'$.
\end{proof}

\begin{figure}[ht]
\centering
\begin{tikzpicture}[x=1cm,y=1cm,>=stealth]
  \draw[->] (-4,0) -- (4.3,0) node[right] {$\theta$};
  \draw (3.5,0.07) -- (3.5,-0.07) node[below] {$\pi$};
  \draw (-3.5,0.07) -- (-3.5,-0.07) node[below] {$-\pi$};
  \draw (1.75,0.07) -- (1.75,-0.07) node[below] {$c$};
  \draw (-1.75,0.07) -- (-1.75,-0.07) node[below] {$-c$};
  \draw (0,0.07) -- (0,-0.07) node[below] {$0$};
  \draw[very thick] (-1.75,1.1) -- (1.75,1.1);
  \node at (0,1.38) {$+\,[\,y_{\mathrm{big}}\text{ on }[-c,c]\,]$};
  \draw[very thick,dashed] (-3.5,1.1) -- (-1.75,1.1);
  \draw[very thick,dashed] (1.75,1.1) -- (3.5,1.1);
  \node at (-2.62,1.38) {$-$};
  \node at (2.62,1.38) {$-$};
  \fill (1.75,1.1) circle (1.2pt);
  \fill (-1.75,1.1) circle (1.2pt);
  \draw[thick,->] (1.65,-0.9) -- (-1.65,-0.9);
  \node at (0,-1.28) {$\alpha_2:\ P\to\bar P$};
  \draw[thick,->] (1.85,-0.9) -- (3.45,-0.9);
  \draw[thick,->] (-3.45,-0.9) -- (-1.85,-0.9);
  \node at (2.62,-1.28) {$\alpha_1:\ -P\to-\bar P$};
  \node at (-2.62,-1.28) {$\alpha_1$};
\end{tikzpicture}
\caption{The closed anti-invariant cycle $C'=\tilde\gamma+\beta_0$ on the
torus $|x|=1$ (schematic). Top: Samart's signed open chain $\tilde\gamma$ on
the big branch (weight $+$ on $[-c,c]$, weight $-$ on the outer arcs), with
branch jumps at the fold points $\theta=\pm c$. Bottom: the compensating
small-branch arcs $\beta_0=\alpha_1+\alpha_2$; their boundary is the negative
of $\partial\tilde\gamma$, so $C'$ is closed, and complex conjugation shows
$c(C')=-C'$.}
\label{fig:chain}
\end{figure}
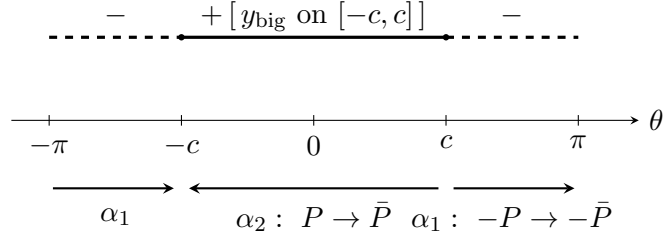

\subsection{Homology class: \texorpdfstring{$C'=2\gamma^-$}{C'=2g-} (certified)}
\label{subsec:class}

In genus $1$ with $\Delta<0$ the period lattice has a basis $(w_1,w_2)$
with $w_1$ real and $\Re(w_2/w_1)=1/2$ (one real component); for
\texttt{11.a3}, PARI \texttt{E.omega} gives $\tau=w_2/w_1=1/2+0.2298780212\ldots i$
(in the upper-half-plane convention). The anti-invariant homology is then
explicit:

\begin{lemma}[The anti-invariant generator]\label{lem:antiinv}
Let $E/\R$ have $\Delta<0$, with period basis $(w_1,w_2)$ as above, and let
$(a,b)$ be the dual basis of $H_1(E,\Z)$. Complex conjugation acts by
\[
  a\mapsto a,\qquad b\mapsto a-b,
  \qquad\text{i.e.}\qquad
  \begin{pmatrix}1&1\\0&-1\end{pmatrix}.
\]
Consequently
\[
  H_1(E,\Z)^-=\Z\,(a-2b),
\]
the generator $a-2b$ is primitive, and its period
$\period(a-2b)=w_1-2w_2$ is the unique (up to sign) purely imaginary period
of the lattice.
\end{lemma}

\begin{proof}
Conjugation fixes the real period $w_1$, hence the class $a$, and sends the
loop $b$ to a loop of period $\overline{w_2}=w_1-w_2$, hence to the class
$a-b$; the displayed matrix is an involution. A class $\alpha a+\beta b$ is
anti-invariant iff $(\alpha+\beta)a-\beta b=-\alpha a-\beta b$, i.e.\
$\beta=-2\alpha$, so $H_1(E,\Z)^-=\Z(a-2b)$; $a-2b$ is primitive since
$\gcd(1,2)=1$. Its period $w_1-2w_2=-2i\,\Im w_2$ is purely imaginary, and
conversely a purely imaginary period $mw_1+nw_2$ must satisfy $2m+n=0$,
hence is an integral multiple of $w_1-2w_2$.
\end{proof}

The pairing $\gamma\mapsto\period(\gamma)$ identifies $H_1(E,\Z)$ with the
period lattice, so the pairing with $\omega$ is injective. We fix the sign
convention
{\footnotesize
\[
  w_{\mathrm{anti}}:=-2.9176332338769904586617792258073505143184868579\ldots\,i
  \qquad(\text{PARI \texttt{11.a3}}),
\]}
and choose $\gamma^-$ with $\period(\gamma^-)=w_{\mathrm{anti}}$. Therefore
$\period(C')/w_{\mathrm{anti}}$ is \emph{a priori} an integer.

The computation (\texttt{code/n1\_certify.py}, mpmath at $60$ and $80$
digits, cross-checked against PARI \texttt{11.a3} \texttt{E.omega}) gives
\begin{equation}\label{eq:period}
  \period(C')=I_{\mathrm{signed}}+A_{s,\mathrm{outer}}-A_{s,\mathrm{inner}}
  =-5.8352664677539809\ldots\,i,
\end{equation}
\[
  \frac{\period(C')}{w_{\mathrm{anti}}}=1.9999999999999999\ldots,\qquad
  \big|\period(C')-2w_{\mathrm{anti}}\big|=2.5\times10^{-16},
\]
where $I_{\mathrm{signed}}=-2.917633233876990458\ldots i=w_{\mathrm{anti}}$
is the period of the open signed chain and
$A_{s,\mathrm{outer}},A_{s,\mathrm{inner}}$ are the period integrals of the
compensating arcs; the $60$- and $80$-digit runs agree, the residual error
coming from the $1/\sqrt\theta$ endpoint singularity at $\theta=0$. A-priori
integrality, the distance-to-nearest-integer $2.5\times10^{-16}\ll 1$, and
the ball-arithmetic certification of \S\ref{subsec:rigour} imply
\[
  \boxed{\class(C')=2\,\gamma^-}.
\]
(As a by-product one obtains the two independent period identities
$I_{\mathrm{signed}}=w_{\mathrm{anti}}$ and
$A_{s,\mathrm{inner}}-A_{s,\mathrm{outer}}=-w_{\mathrm{anti}}$.)

\begin{remark}
The open chain $\tilde\gamma$ of \S\ref{sec:proof} is not itself closed
and cannot be assigned a homology class; the closed cycle $C'$ used here
has winding number $2$. For
contrast, the naive loop integral $I_{\mathrm{loop}}=-0.47447\ldots i$ is not
even the period of any integral cycle (ratio $0.16262\ldots$ to
$w_{\mathrm{anti}}$).
\end{remark}

\subsection{The regulator identity along the closed chain}

On $|x|=1$, $\eta(x,y)=\log|x|\,d\arg y-\log|y|\,d\arg x=-\log|y|\,d\theta$,
and since the product of the two roots is $x^3$ of modulus $1$,
$\log|y_{\mathrm{big}}|=-\log|y_{\mathrm{small}}|$ pointwise. Writing
$J_1=\int_0^c\log|y_s|\,d\theta$ and $J_2=\int_c^\pi\log|y_s|\,d\theta$,
direct computation gives
\[
  \int_{\tilde\gamma}\eta=2(J_1-J_2)
  =\underbrace{(2J_1)}_{\alpha_2}+\underbrace{(-2J_2)}_{\alpha_1}
  =\int_{\beta_0}\eta,
\]
hence the \emph{exact} identity
\begin{equation}\label{eq:double}
  \int_{C'}\eta=2\int_{\tilde\gamma}\eta.
\end{equation}

\begin{lemma}[Pairing with homology]\label{lem:homology}
The integral $\int_{C'}\eta(x,y)$ depends only on the class of $C'$ in
$H_1(E,\Z)^-$.
\end{lemma}

\begin{proof}
The face polynomials of the cubic $y^2+y=x^3-x^2$ are all cyclotomic, so
by the temperedness criterion of Rodriguez~Villegas \cite{RV} the symbol
$\{x,y\}$ is tempered: all its tame symbols $T_v(\{x,y\})$ are roots of
unity (verified exactly in \texttt{code/bertin\_diamond.py}; cf.\
\S\ref{sec:proof}).  The \emph{real} residue of $\eta$ at each point $v$
of the support is therefore $\pm\log|T_v(\{x,y\})|=0$: $\eta$ carries no
delta-mass residues, is closed on the complement of the support, and has
only (integrable) logarithmic singularities at the support points.  It
follows that the pairing with $\eta$ descends to homology, and thence by
anti-invariance to $H_1(E,\Z)^-$ (Lemma~\ref{lem:antiinv}).  The chains
used here are admissible: $C'$ avoids
$\operatorname{supp}\divisor(x)\cup\operatorname{supp}\divisor(y)$, since
those points have $x\in\{0,\infty\}$ while $C'\subset\{|x|=1\}$; and at
the fold points $\log|y|=0$, so $\eta=-\log|y|\,d\theta$ is continuous
and integrable along $C'$.
\end{proof}

\subsection{Regulator constants: Bloch, Brunault, Bertin}
\label{subsec:anchor}

The divisor algebra (projective computation on the minimal cubic model
$y^2+y=x^3-x^2$, to which the data are transported by the explicit birational
change of variables; note that the two points at infinity of $S_0$ map to
\emph{different} points, $Q_\infty\mapsto O$ and $[0:1:0]\mapsto 3A$; Abel's
principal-divisor test passed) gives, with
$A=(0,0)$ the $5$-torsion point and $O=[0:1:0]$,
\[
  \divisor(x)=[A]+[2A]-[O]-[3A],\qquad
  \divisor(y)=3[2A]-2[3A]-[O],
\]
hence the diamond product
\[
  (x)\diamond(y)=6(O)-5(A)+5(2A),\qquad
  D_E\big((x)\diamond(y)\big)=-5D_E(P)+5D_E(2P),
\]
where we write $P:=A=(0,0)$ (Brunault's notation) when the point appears as
an argument of $D_E$.
The elliptic dilogarithm values (mpmath, $60$ digits; lattice basis
$(w_1,w_1-w_2)$ so that $\tau=0.5+0.2299i$; $z(P)=0.6w_1$, $z(2P)=0.2w_1$
from PARI \texttt{ellpointtoz}):
\[
  D_E(P)=0.1911937370843316957549544343121738161012\ldots
\]
The two key identities (both verified to $40$ digits, and both theorems in
the literature) are
\begin{itemize}
  \item \textbf{Brunault's coefficient} \cite[(3.151)]{Brunault}:
  $D_E(P)=\dfrac{2\pi}{5}b_{11}=\dfrac{11}{10\pi}L(E,2)$, i.e.\
  $L(E,2)=\frac{10\pi}{11}D_E(P)$;
  \item \textbf{Bertin's exotic relation} \cite{Bertin,BertinCrelle}:
  $D_E(2P)=\frac32D_E(P)$ (note $\tfrac32$, not the naive $2$).
\end{itemize}

The evaluation of the regulator of $\{x,y\}$ along $\gamma^-$ is carried out
\emph{directly}, via Brunault's proved regulator formula for Siegel units
\cite{BrunaultSiegel}; Bloch's diamond theorem is not used.  We then record
the clarification of the factor-of-two question that the computation
entails (Lemma~\ref{lem:coinvariant} and the discussion below), which is
likewise not part of the proof.

\begin{lemma}[The regulator generator and the index]\label{lem:coinvariant}
Let $E/\R$ be an elliptic curve, $c$ the complex-conjugation involution,
$H_1(E,\Z)^\pm=\ker(c_*\mp1)$, and $\{f,g\}$ a tempered symbol on $E$.
The pairing $\gamma\mapsto\int_\gamma\eta(f,g)$ descends to the quotient
$H_1(E,\Z)/H_1(E,\Z)^+$ of the integral homology by the invariant
sublattice (throughout, ``coinvariant quotient'' means exactly this
quotient by $H_1(E,\Z)^+$; we do not use $\ker(c_*+1)$ and the quotient
interchangeably, as they differ by an index that is material here).
\begin{enumerate}
  \item If $\Delta>0$, then $H_1(E,\Z)^-=Z\,b$ maps isomorphically onto
  $H_1(E,\Z)/H_1(E,\Z)^+=Z\,[b]$: the anti-invariant generator represents
  the quotient generator.
  \item If $\Delta<0$, the image of $H_1(E,\Z)^-=Z\,\gamma^-$ in
  $H_1(E,\Z)/H_1(E,\Z)^+=Z\,[b]$ is $2\,Z\,[b]$, of index $2$: with the
  basis of Lemma~\ref{lem:antiinv}, $\gamma^-=a-2b\equiv-2[b]$, and hence
  for every tempered symbol
  \[
    \int_{\gamma^-}\eta(f,g)=-2\int_{[b]}\eta(f,g).
  \]
\end{enumerate}
\end{lemma}

\begin{proof}
The pairing is trivial on $H_1(E,\Z)^+$: if $c(\gamma)=\gamma$ then
$\int_\gamma\eta=\int_{c(\gamma)}\eta=\int_\gamma c^*\eta=-\int_\gamma\eta$
since $c^*\eta=-\eta$ (cf.\ \cite[Rem.~4]{LalinRam}); hence it descends to
the coinvariants.  For $\Delta<0$, $c_*:a\mapsto a$, $b\mapsto a-b$
(Lemma~\ref{lem:antiinv}), so $H_1(E,\Z)^+=Z\,a$ and
$H_1(E,\Z)/H_1(E,\Z)^+=Z\,[b]$; the anti-invariant class $a-2b$ projects
to $-2[b]$, giving index $2$.  For $\Delta>0$, $c_*:a\mapsto a$,
$b\mapsto-b$, so $H_1^-=Z\,b$ and the projection $b\mapsto[b]$ is an
isomorphism.
\end{proof}

\paragraph{The anchor: a direct Siegel-unit computation.}
The symbol $\{x,y\}$ is a pair of modular units (\S\ref{sec:proof}, Modular
units: its divisors are supported on the rational cusps of $X_1(11)$).
Regulator integrals of Siegel units are computed by a proved formula that
involves neither the diamond product nor a choice of homology lattice:
for Siegel units $g_u,g_v$, $u=(a,b)$, $v=(c,d)\in(\Z/N\Z)^2\setminus0$,
\begin{equation}\label{eq:siegelthm}
  \int_0^{i\infty}\eta(g_u,g_v)
  \;=\;\pi\,\Lambda^*\bigl(e_{a,d}\,e_{b,-c}+e_{a,-d}\,e_{b,c},\,0\bigr),
\end{equation}
with $e_{a,b}$ the explicit weight-$1$ level-$N^2$ Eisenstein series of
\cite[Thm.~1, eq.~(2)]{BrunaultSiegel} and $\Lambda^*$ the regularized
completed $L$-value of \cite[\S2]{BrunaultSiegel}; arbitrary modular
symbols are handled by linearity \cite[Rem.~2]{BrunaultSiegel}.  (We
verified the factor $\pi$ in \eqref{eq:siegelthm} against the article's
source, and validated our implementation of the formula by reproducing the
conductor-$14$ application of \cite[\S5.1]{BrunaultSiegel} to $60$ digits
by three independent methods; archive
\texttt{notes/attack16-siegel-anchor.txt}.)

\begin{theorem}[The regulator anchor]\label{thm:anchor}
With $\gamma^-$ the anti-invariant generator of Lemma~\ref{lem:antiinv},
\begin{equation}\label{eq:anchor}
  \int_{\gamma^-}\eta(x,y)=\pm\,2\pi b_{11},
\end{equation}
the sign depending on orientations.
\end{theorem}

\begin{proof}
All steps are exact rational/$q$-expansion computations or are archived
(\texttt{code/siegel\_anchor\_step1}--\texttt{14},
\texttt{notes/attack16-siegel-anchor.txt},
\texttt{notes/attack17-membership.txt},
\texttt{notes/attack17-primitivity.txt},
\texttt{notes/attack17-constants.txt}).
\emph{(1) Cusps and divisors.}  Under the modular parametrization
$\pi\colon X_1(11)\to E$ attached to $f_{11}$ (infinite cusp $\mapsto O$),
the rational cusps map as $k/11\mapsto m_kA$ with
\begin{equation}\label{eq:mtable}
  (m_1,\dots,m_5)=(0,2,1,4,3).
\end{equation}
This is derived exactly, not numerically.  Brunault, after Lecacheux,
tabulates the five rational cusps $P_v$ of $X_1(11)$, indexed by
$v\in(\Z/11\Z)^\times/\{\pm1\}$, as points of $E$
\cite[(3.152)--(3.153)]{Brunault}: $P_1=\infty$, $P_2=(1,0)$,
$P_3=(0,-1)$, $P_4=(0,0)$, $P_5=(1,-1)$, with $P_{4^a}=a\,P_4$ and $P_4$
a generator of $E(\Q)\cong\Z/5\Z$.  With the group law $2A=(1,-1)$,
$3A=(1,0)$, $4A=(0,-1)$ for $A=(0,0)$ this reads $P_v=n(v)A$,
$n=(0,3,4,1,2)$.  The label conversion is the \emph{inverse} one, and it
follows directly from Brunault's definition
$P_v=\langle v\rangle i\infty=[0,v]$: a matrix
$\gamma=\bigl(\begin{smallmatrix}k&b\\11&d\end{smallmatrix}\bigr)
 \in\mathrm{SL}_2(\Z)$
represents the cusp $k/11$, and the determinant $kd-11b=1$ gives
$d\equiv k^{-1}\pmod{11}$; hence the bottom-row label of $k/11$ is
$[0,d]=[0,k^{-1}]$, that is,
\begin{equation}\label{eq:labelconv}
  k/11 \;=\; P_{k^{-1}}
\end{equation}
up to the $\pm1$ convention.  No divisor comparison enters this
identification.  Together with the group-law reading above,
\eqref{eq:labelconv} gives \eqref{eq:mtable}.  Hence
$x\circ\pi,y\circ\pi$ have cusp orders
\[
  \operatorname{ord}_{k/11}(x\circ\pi)=(-1,+1,+1,0,-1),\qquad
  \operatorname{ord}_{k/11}(y\circ\pi)=(-1,+3,0,0,-2)
\]
for $k=1,\dots,5$, matching exactly the divisors displayed at the
beginning of this subsection; conversely, these orders together with the
Kubert--Lang orders of \emph{(2)} force the tuple \eqref{eq:mtable}
uniquely, an independent exact check that does not enter the
identification above.  The $60$-digit Abel-integral
evaluation of the cusp images agrees with both derivations and is used
only as a check (\texttt{siegel\_anchor\_step14.py},
\texttt{notes/attack17-constants.txt}).
\emph{(2) Siegel presentations.}  With
$G_a:=\prod_{b\bmod 11}g_{a,b}$ and the Kubert--Lang cusp orders
$\operatorname{ord}_{(r,t)}g_{a,b}=\tfrac{11}{2}B_2(\{(ar+bt)/11\})$ on the
$60$ cusps of $X(11)$, exact rational linear algebra gives
\begin{equation}\label{eq:siegelpres}
  x\circ\pi=-\,\frac{G_4G_5}{G_2^2},
  \qquad
  y\circ\pi=\frac{G_1G_5^{\,3}}{G_2^{\,3}G_3}.
\end{equation}
Each side of \eqref{eq:siegelpres} is $\Gamma_1(11)$-invariant with the
same cusp divisor, so each ratio is a nonzero constant:
$x\circ\pi=C_x\,U$, $y\circ\pi=C_y\,V$ with $U=G_4G_5/G_2^2$,
$V=G_1G_5^3/(G_2^3G_3)$.  The constants are determined exactly.  At a
cusp where both sides have order $0$ one may simply evaluate: by
\eqref{eq:mtable}, $\pi(4/11)=4A$ and $\pi(3/11)=A$ are such cusps for
$U$ and $V$ respectively; the leading $q$-coefficients $\kappa_c$ of
$U,V$ at a cusp $c$ are explicit roots of unity (write
$\gamma i\infty=c$ as a word in $S,T$ and apply
\cite[Lemma~4 and eq.~(3)]{BrunaultSiegel}; the exact cyclotomic
computation gives $\kappa_{4/11}(U)=\kappa_{3/11}(V)=-1$), and
$x(4A)=1$, $y(A)=-1$, hence
\[
  C_x=\frac{x(4A)}{\kappa_{4/11}(U)}=\frac{1}{-1}=-1,\qquad
  C_y=\frac{y(A)}{\kappa_{3/11}(V)}=\frac{-1}{-1}=+1
\]
(\texttt{siegel\_anchor\_step14.py}).  Because the constants have modulus
one, their correction to the regulator integral vanishes: since
$\eta(CU,C'V)=\eta(U,V)+\log|C|\,d\!\arg V-\log|C'|\,d\!\arg U$, one has
$\int_{\gamma^-}\eta(x\circ\pi,y\circ\pi)=\int_{\gamma^-}\eta(U,V)
 +\log|C_x|\,D_V-\log|C_y|\,D_U$ with $D_F:=\int_{\gamma^-}d\!\arg F$,
and Brunault's Lemma~5 \cite{BrunaultSiegel},
\[
  \int_0^{i\infty}d\!\arg g_{a,b}=
  \begin{cases}
    0 & \text{if }a\equiv0\text{ or }b\equiv0\pmod{11},\\[1mm]
    2\pi\bigl(\{\tfrac{a}{11}\}-\tfrac12\bigr)
        \bigl(\{\tfrac{b}{11}\}-\tfrac12\bigr) & \text{otherwise},
  \end{cases}
\]
evaluates each piece of each of the seven Manin symbols of \emph{(3)} as
an explicit rational multiple of $2\pi$ (the factors $w$ in
$g_{a,b}\circ\gamma=w\,g_{(a,b)\gamma}$ are constant, hence invisible to
$d\!\arg$); exact summation gives
\begin{equation}\label{eq:DUDV}
  D_U=2\pi,\qquad D_V=0
\end{equation}
(\texttt{siegel\_anchor\_step14.py}).  Note that $D_U=2\pi\neq0$: $U$
has winding number $1$ along $\gamma^-$, so the vanishing of the
correction is not automatic --- it holds because
$\log|C_x|=\log|C_y|=0$ exactly.
\emph{(3) The cycle.}  Take the modular symbol
$\gamma^-=\{0,\tfrac{3}{11}\}-\{0,\tfrac{8}{11}\}$: it is closed on
$X_1(11)$ (as $3\equiv-8\bmod 11$) and anti-invariant under
$\tau\mapsto-\bar\tau$.  Its primitivity is proved by exact integral
linear algebra (\texttt{siegel\_anchor\_step13.py},
\texttt{notes/attack17-primitivity.txt}): for Manin symbols of
$\pm\Gamma_1(11)$ (index $60$ in $\mathrm{PSL}_2(\Z)$, parametrized by
bottom rows $(c,d)\in((\Z/11\Z)^2\setminus0)/\pm1$) subject to $x+xS=0$
and $x+xR+xR^2=0$, with boundary map to the $10$ cusps, a Smith normal
form computation gives $H_1(X_1(11),\Z)=\ker\partial\cong\Z^2$
torsion-free with an explicit integral basis; conjugation induces
$C=\bigl(\begin{smallmatrix}0&1\\1&0\end{smallmatrix}\bigr)$ on $H_1$,
so $H_1^-=\ker(C+I)=\Z\cdot(-1,1)$; and the seven-symbol chain below
has coordinates $(1,-1)$ in this basis --- exactly $\pm1$ times a
generator of $H_1^-$ (equivalently, it has intersection number $\pm1$
with an integral class).  Hence it is the class $\pm(a-2b)$ of
Lemma~\ref{lem:antiinv}.  (Its period equals $w_{\mathrm{anti}}$ to $60$
digits, and PARI's modular-symbol normalization takes the exact value
$v^-=1$ on it; both are checks only.)  Continued fractions decompose it
into seven Manin symbols:
\[
\begin{aligned}
  \{0,\tfrac{3}{11}\}&=\textstyle
   +\bigl[\begin{smallmatrix}1&0\\3&1\end{smallmatrix}\bigr]
   -\bigl[\begin{smallmatrix}1&1\\3&4\end{smallmatrix}\bigr]
   +\bigl[\begin{smallmatrix}3&1\\11&4\end{smallmatrix}\bigr],\\
  \{0,\tfrac{8}{11}\}&=\textstyle
   +\bigl[\begin{smallmatrix}1&0\\1&1\end{smallmatrix}\bigr]
   -\bigl[\begin{smallmatrix}1&2\\1&3\end{smallmatrix}\bigr]
   +\bigl[\begin{smallmatrix}3&2\\4&3\end{smallmatrix}\bigr]
   -\bigl[\begin{smallmatrix}3&8\\4&11\end{smallmatrix}\bigr].
\end{aligned}
\]
\emph{(4) Evaluation.}  Applying \eqref{eq:siegelthm} term by term to
\eqref{eq:siegelpres} along these symbols expresses the integral as
$\pi\,\Lambda^*(F_{\mathrm{total}},0)$, where $F_{\mathrm{total}}$ is an
explicit finite signed sum of products
$e_{a,d}e_{b,-c}+e_{a,-d}e_{b,c}$.  Each $e_{a,b}$ is an Eisenstein
series of weight $1$ on $\Gamma_1(121)$
\cite[Definition~10 and Lemma~11]{BrunaultSiegel}, holomorphic on
$\mathcal{H}$ and at every cusp; hence the \emph{unconditional}
membership $F_{\mathrm{total}}\in M_2(\Gamma_1(121))$, and
$D:=F_{\mathrm{total}}+2f_{11}\in M_2(\Gamma_1(121))$.  The Sturm bound
for this space is
$\frac{2}{12}\,[\mathrm{PSL}_2(\Z):\bar\Gamma_1(121)]
 =\frac{2}{12}\cdot 7260=1210$ (the index being
$\frac{121^2}{2}(1-\frac1{121})=7260$, as $-I$ acts trivially in even
weight).  Exact rational $q$-expansion arithmetic (the higher
$e$-coefficients are integral and
$\alpha_0(a,b)\in\frac1{22}\Z$; the convolutions are evaluated exactly)
gives $a_n(D)=0$ for all $0\le n\le 2420$ --- twice the sharp bound, and
in particular through the conservative $\mathrm{SL}_2$-index convention
$\frac{2}{12}\cdot14520=2420$ (\texttt{siegel\_anchor\_step12.py},
\texttt{notes/attack17-membership.txt}).  By Sturm's theorem $D=0$,
that is,
\[
  F_{\mathrm{total}}=-2\,f_{11}
\]
as modular forms; a fortiori $F_{\mathrm{total}}\in M_2(\Gamma_0(11))$.
(The same Sturm computation applied symbol by symbol certifies that the
symbols $\bigl[\begin{smallmatrix}1&0\\3&1\end{smallmatrix}\bigr]$ and
$\bigl[\begin{smallmatrix}1&2\\1&3\end{smallmatrix}\bigr]$ contribute
$-f_{11}$ each and that the remaining five vanish identically; the main
argument uses only the total.)  Therefore the
$\Lambda$-sum equals $-2\,\Lambda(f_{11},0)=-2\,b_{11}$, since with
$\Lambda(f,s)=11^{s/2}(2\pi)^{-s}\Gamma(s)L(f,s)$ and $L(f_{11},0)=0$
(root number $+1$) one has $\Lambda(f_{11},0)=L'(f_{11},0)=b_{11}$.
Multiplying by the factor $\pi$ of \eqref{eq:siegelthm} gives
$\int_{\gamma^-}\eta(x,y)=-2\pi b_{11}$ with the orientation of (3),
i.e.\ \eqref{eq:anchor} in general.  Independent corroboration: the
numerical $\Lambda$-sum is $-2b_{11}$ to $50$ digits
(\texttt{siegel\_anchor\_step8.py}), and direct numerical integration of
$\eta(x,y)$ along $\gamma^-$ with the opposite orientation gives
$+0.9559686854216584787\ldots$ to $45$ digits
(\texttt{siegel\_anchor\_step11.py}).
\end{proof}

The regulator side of (C3) is therefore a proved theorem, independent of
Bloch's diamond formula; the remaining gap, closed in \S\ref{subsec:class},
is the identification of Boyd's split-integral chain with the closed cycle
$C'=2\gamma^-$.

\paragraph{The status of Bloch's diamond formula (not used).}
For orientation we record what the computation implies for Bloch's theorem.
As transmitted by \cite[Thm.~6]{LalinRam}, that theorem states
\begin{equation}\label{eq:bloch}
  \int_{\gamma}\eta(f,g)
  =D_E\big((f)\diamond(g)\big)
\end{equation}
for $\gamma$ a generator of the anti-invariant \emph{subgroup}
$H_1(E,\Z)^-$, with $D_E$ given by \cite[Def.~5, eq.~(10)]{LalinRam} (that
series is, term by term, the series we implement; a non-rigorous numerical
comparison on both curves agrees to $60$ digits).  On the present curve
($\Delta<0$) the subgroup reading of \eqref{eq:bloch} with factor $1$ is
incompatible with the proved value of the anchor:
$D_E((x)\diamond(y))=\frac52D_E(P)=\pi b_{11}$ while
$\int_{\gamma^-}\eta(x,y)=\pm2\pi b_{11}$.  Every verified datum --- the
anchor above, Brunault--Bertin's value for $\{x_W,y_W\}$
(Remark~\ref{rem:brunault-dict}), direct numerical integration on both
generators (Remark~\ref{rem:diamondk0}), and the proved conductor-$17$
computation of \cite{LalinRam}, where the subgroup and the quotient
coincide (Lemma~\ref{lem:coinvariant}~(1)) --- is instead consistent with
\eqref{eq:bloch} holding with factor $1$ for a generator $\bar\gamma$ of
the coinvariant quotient $H_1(E,\Z)/H_1(E,\Z)^+$, the group on which the
regulator is naturally defined (both sides being extended $\Q$-linearly in
the symbol); by Lemma~\ref{lem:coinvariant}~(2) the subgroup generator then
reads twice as much, exactly as observed.  We record this neutrally as a
\emph{normalization/lattice discrepancy} --- subgroup versus
coinvariant-quotient generator, related by the index $2$ of
Lemma~\ref{lem:coinvariant}~(2) --- which the present data resolve in
favour of the quotient reading; whether this reflects a misnormalization
in the transmitted subgroup statement of \cite[Thm.~6]{LalinRam} (a
statement their paper never uses) or a convention we have not
reconstructed, we leave as an open question for the literature.  We
emphasize once more that nothing in this paper depends on the answer.

\begin{remark}[The factor $2$ in the $\diamond$-form of Bloch's theorem,
resolved]\label{rem:diamondk0}
The formal expansion evaluates
$D_E((x)\diamond(y))=\frac52D_E(P)=\pi b_{11}$, while the certified integral
over the subgroup generator is $\int_{\gamma^-}\eta=2\pi b_{11}$: in the
factor-$1$ normalization $r(\{x,y\})[\gamma]=D^E((x)\diamond(y))$,
$r[\gamma]=\int_\gamma\eta$, of \cite[Thm.~6]{LalinRam} the two sides differ
by exactly $2$. We investigated this discrepancy with three \emph{different}
tempered symbols on the same curve (\texttt{code/bertin\_diamond.py},
\texttt{code/bertin\_diamond.gp}, archive
\texttt{notes/attack13-bertin-diamond.txt}): the Weierstrass symbol above,
our $S_0$-symbol, and the coordinate symbol $\{X,Y\}$ of Bertin's cubic
$(X+1)(Y+1)(X+Y+1)+XY=0$, which maps to $E$ by an explicit Riemann--Roch
transformation (PARI \texttt{ellidentify} certifies the image model as
\texttt{11.a3}). The $\diamond$-values are $-\frac52$, $+\frac52$ and
$\frac{35}{2}$ times $D_E(P)$ respectively, while direct integration of
$\eta$ along $\gamma^-$ gives $-5$, $+5$ and $+35$ times $D_E(P)$ (the first
two certified as above; the third by extrapolated Riemann sums, consistent
with $2\pi\,m(C_1)=14\pi b_{11}$). The factor is exactly $2$ in all three
cases: it is a property of $(E,\gamma^-)$, not of the symbol, the
$\diamond$-convention (identical on both curves), the tame symbols (all
roots of unity), or the $D_E$ series.
The resolution is the index phenomenon of Lemma~\ref{lem:coinvariant}. The
regulator pairing of \cite[Def.~3]{LalinRam} vanishes on $H_1(E(\C),\Z)^+$
(their Remark~4), hence factors through the \emph{coinvariant quotient}
$H_1/H_1^+$; Bloch's theorem \eqref{eq:bloch} therefore evaluates the
integral over the \emph{quotient} generator $[\gamma^-_0]=[b]$ with
factor $1$, and the integral over the \emph{subgroup} generator
$\gamma^-=a-2b$ acquires the index $2$ of $\Z\cdot\gamma^-$ in $H_1/H_1^+$
(Lemma~\ref{lem:coinvariant}(2)). On the conductor-$17$ curve the index is
$1$ (Lemma~\ref{lem:coinvariant}(1): $\Delta>0$, $b\mapsto-b$), the subgroup
and quotient coincide, and the same computation closes with factor $1$ ---
corroborated there by the proved theorem of \cite{LalinRam}. The three
symbols' uniform factor $2$ is thus exactly the $\Delta<0$ index, and the
conductor-$17$ factor $1$ its $\Delta>0$ counterpart.
We verified the quotient reading numerically on the conductor-$11$ curve
itself (\texttt{code/verify\_coinvariant.gp}, archive
\texttt{notes/attack15-coinvariant.txt}): integrating $\eta(x_W,y_W)$ along
cycles shifted to avoid its poles, with a centered discretization whose
convergence order is measured empirically to be $2.000000$ at every
doubling from $N=500$ to $N=8000$,
\[
  \textstyle
  \int_{a}\eta=0,\qquad
  \int_{b}\eta=-\pi b_{11}
   =D^E\bigl((x_W)\diamond(y_W)\bigr),\qquad
  \int_{a-2b}\eta=2\pi b_{11},
\]
where the first integral converges to $0$ spectrally (below $10^{-75}$, the
working-precision floor, from $N=2000$ on), and the raw errors of the
second and third decrease like $N^{-2}$ from $-1.4\cdot10^{-5}$ resp.\
$-6.5\cdot10^{-6}$ at $N=500$ to $-5.4\cdot10^{-8}$ resp.\
$-2.5\cdot10^{-8}$ at $N=8000$; a single $p=2$ Richardson step on the
finest pair leaves residuals $-1.8\cdot10^{-15}$ and $-9.2\cdot10^{-16}$
against $-\pi b_{11}$ and $2\pi b_{11}$, and
$-2\int_b\eta-\int_{a-2b}\eta=4.6\cdot10^{-15}$ on the extrapolants.
We stress that this is a numerical consistency check, not a proof.
Subject to that caveat, the factor-$1$ identity \eqref{eq:bloch} holds for
the symbol
$\{x_W,y_W\}$ on the quotient generator, and its subgroup-generator value is
the doubled \eqref{eq:anchor}.
A by-product: $D_E((X)\diamond(Y))=7\,D_E((x)\diamond(y))$ exactly, which
explains the coefficient $7$ in Bertin's $m(C_1)=7b_{11}$ at the level of
$K_2(E)$. Finally, the form $\pi r=D_E(\cdot)$ found in expository accounts
(e.g.\ \cite[Thm.~1]{Touafek}) is consistent with neither the factor-$1$
nor the factor-$2$ data above and is not used anywhere in this paper.
\end{remark}

\begin{remark}[Cross-check against Brunault's symbol dictionary]
\label{rem:brunault-dict}
Our proof uses Bloch's theorem only through its proved instance on the
conductor-$17$ curve (Appendix~\ref{app:k1}); the following is a
cross-check, not part of the proof.
Brunault's Theorem~3.9.3.118 \cite{Brunault} evaluates
$r_{\gamma^-}\{x,y\}$ for Weierstrass coordinates $x,y$ on $X_1(11)$ in
terms of $D_E$, and its proof concludes with (3.210)--(3.211):
\begin{align}
  r_{\gamma^-}\{x_W,y_W\}
  &=\frac{1}{2\pi}\,D^E\bigl(8(O)+5(A)-5(2A)\bigr)
  =-\frac{5}{2\pi}\,D_E(P),\label{eq:dict1}\\
  \bigl|r_{\gamma^-}\{x_W,y_W\}\bigr|
  &=\frac{5}{2\pi}\,D_E(P)\;=\;b_{11},\label{eq:dict2}
\end{align}
where \eqref{eq:dict2} substitutes Brunault's own Corollaire~3.5.101
$\zeta$-evaluation $L'(E,0)=5D_E(P)$ \cite[(3.151)]{Brunault} into his
(3.211). We checked the identification of
\eqref{eq:dict1}--\eqref{eq:dict2} with our $\{x_W,y_W\}$ in three
independent ways: textually (the proof states that $x,y$ are the pull-backs
of ``coordonn\'ees de Weierstrass'' via $j^*$ ``sur un mod\`ele de
Weierstrass''); divisor-wise ($\divisor x_W=[A]+[4A]-2[O]$,
$\divisor y_W=2[A]+[3A]-3[O]$, and
$(x_W)\diamond(y_W)\equiv 8(O)+5(A)-5(2A)$ modulo $\Z[E(\Q)_{\mathrm{tors}}]$); and
numerically, $5D_E(P)/(2\pi)=b_{11}$ to all $60$ computed digits against
the certified value of \S\ref{sec:cert}. Since Brunault's $\gamma$ is the \emph{subgroup} generator
(his footnote~2: $H_1(E(\C),\Z)^-=\Z\cdot\gamma$), \eqref{eq:dict2}
reads $\bigl|\int_{\gamma^-}\eta(x_W,y_W)\bigr|=2\pi b_{11}$, which is
exactly Lemma~\ref{lem:coinvariant}(2) applied to the quotient value
$\int_{[\gamma^-_0]}\eta(x_W,y_W)=D^E((x_W)\diamond(y_W))=-\pi b_{11}$: the
dictionary, the Bloch anchor, and our numerics all agree once the index is
taken into account.
\end{remark}

\subsection{Synthesis}

\begin{theorem}[(C3)]\label{thm:C3}
$I_{\mathrm{split}}=b_{11}$.
\end{theorem}

\begin{proof}
By \eqref{eq:double} and $\class(C')=2\gamma^-$
(\S\ref{subsec:class}), and the anchor \eqref{eq:anchor},
\[
  \int_{\tilde\gamma}\eta=\frac12\int_{C'}\eta
  =\frac12\cdot 2\cdot(\pm 2\pi b_{11})=\pm 2\pi b_{11},
\]
so $|I_{\mathrm{split}}|=b_{11}$ by the identity
$I_{\mathrm{split}}=\frac{1}{2\pi}\int_{\tilde\gamma}\eta$ of
\S\ref{sec:proof}. Two sign facts settle the remaining choice. First,
$b_{11}=\frac{11}{4\pi^2}L(E,2)>0$: the Euler product of $L(E,s)$
converges absolutely at $s=2$, and every factor
$(1-a_p p^{-2}+p^{-1})^{-1}$ is positive since
$1-a_p p^{-2}+p^{-1}\ge 1-2p^{-3/2}+p^{-1}>0$ by the Weil bound
$|a_p|\le2\sqrt p$. Second, a certified interval enclosure
(ball arithmetic, adaptive bisection with rigorous range bounds;
\texttt{code/sign\_certify.py}, certificate \texttt{notes/attack14-sign-k0.txt})
gives
\[
  I_{\mathrm{split}}\in[\,0.1489,\ 0.1553\,]\subset(0,\infty).
\]
Hence $I_{\mathrm{split}}=+b_{11}$. \qedhere
\end{proof}

\section{Numerical certification}
\label{sec:cert}

\subsection{Certified computations}

All floating-point computations are performed with mpmath at $60$--$300$
digits and cross-checked with PARI/GP 2.15.5:

\begin{itemize}
  \item \textbf{The constant $b_{11}$} (\texttt{code/b11.py}): from the exact
  integer coefficients of $f_{11}=\eta(\tau)^2\eta(11\tau)^2$ (truncated
  convolution of the squared Euler function) via the approximate functional
  equation for weight $2$, root number $+1$:
  \[
    b_{11}=\Lambda(f,2)=\sum_{n\ge1}a_n\Big[e^{-t_n}\Big(\frac1{t_n}
    +\frac1{t_n^2}\Big)+E_1(t_n)\Big],\qquad t_n=\frac{2\pi n}{\sqrt{11}},
  \]
  with terms decaying like $e^{-1.894n}$; the tail bound drives the
  truncation ($442$ terms at $800$ digits of working precision).
  \item \textbf{(C3) to $366$ digits} (\texttt{code/attack13\_c3\_300.py},
  archive \texttt{notes/attack13-c3-300.txt}):
  {\small
  \[
  \begin{aligned}
    I_{\mathrm{split}}
    ={}&0.15214714172591804948622729747863449562814\\
       &3589164226122809889823882023289695302776676\ldots,
  \end{aligned}
  \]}
  \[
    |I_{\mathrm{split}}-b_{11}|=9.26\times10^{-367},
  \]
  with the $\eta$-product series value of $b_{11}$ agreeing digit-for-digit
  with PARI/GP \texttt{lfun(E,0,1)} at $330$ digits.
  \item \textbf{(C1)/(C2) reproduced} to $52$ digits
  (\texttt{code/attack1.py}): $\|m-7b_{11}\|\approx5.0\times10^{-53}$,
  $\|m-5b_{11}\|\approx3.6\times10^{-53}$.
  \item \textbf{Period certification} (\texttt{code/n1\_certify.py},
  \texttt{code/n1\_certify.gp}): \eqref{eq:period} computed independently at
  $60$ and $80$ digits with identical results; each arc integral $A_s$ is
  stable to $24$ digits across $40/60/80$-digit runs; $w_{\mathrm{anti}}$
  taken from PARI \texttt{E.omega} for \texttt{11.a3}. The integrality
  argument has an a-priori gap of $1$ against a residual of
  $2.5\times10^{-16}$.
  \item \textbf{Exact algebraic checks}: group law and torsion
  (\texttt{code/torsion.py}) use exact
  rational arithmetic over $\Q$, $\Q(\sqrt2)$, $\Q(\zeta_8)$---no numerical
  approximation. Non-torsion of the endpoint $P$ is proved \emph{rigorously}
  (\texttt{code/endpoint\_torsion3.py}) by reduction at good primes:
  $\operatorname{ord}(P\bmod 17)=20$ and $\operatorname{ord}(P\bmod 89)=3$
  with $\gcd(20,3)=1$, so $P$ cannot have finite order (reduction is
  injective on prime-to-$p$ torsion, Silverman VII.3.1); cross-checked by
  exact $\Q(\zeta_8)$ arithmetic and PARI \texttt{ellorder}. The model
  constant $\kappa=1$ is an \emph{exact rational identity}
  (\texttt{code/kappa\_exact.py}): an explicit birational map to the minimal
  model, derived by Riemann--Roch and composed with PARI's exact minimal
  transform, satisfies $2Y+1=u\,dX/dx$ as a polynomial identity, hence
  $\omega_{\min}=dX/(2Y+1)=dx/u$ exactly (sympy-verified, including the
  inverse round-trip).
  \item \textbf{Branch certification} (\texttt{code/branch\_certify.py},
  output \texttt{notes/attack12-branch.txt}): the eight endpoint branch
  assignments of the closed chain lemma are certified in ball arithmetic at
  $\varepsilon=10^{-6}$ and $10^{-9}$ (certified distance $<1/10$ to the
  claimed endpoint, with $D\ne0$ certified on the closing ball so continuity
  pins the limit), and the modulus ordering of
  Proposition~\ref{prop:structural} is certified by adaptive bisection of
  $[0,\pi]$ ($153$ strict balls with $D\ne0$ and
  $|y_{\mathrm{small}}|<1<|y_{\mathrm{big}}|$, plus the special balls at $0$
  and $\pi/2$ where equality holds exactly).
  \item \textbf{The Bertin-symbol adjudication} (\texttt{code/bertin\_diamond.py},
  \texttt{code/bertin\_diamond.gp}, archive
  \texttt{notes/attack13-bertin-diamond.txt}):
  \begingroup\sloppy
  all divisor, diamond-product
  and tame-symbol computations of \S\ref{sec:proof} are exact (sympy rational
  arithmetic, Abel's principal-divisor test passed); the Riemann--Roch map
  from Bertin's cubic to $E$ is certified by PARI \texttt{ellidentify}
  (\texttt{11.a3}, exact change of variables $[1,-1,-2,2]$); the $D_E$-values
  and the quotient-generator integral $\int_{[\gamma^-_0]}\eta=-\pi b_{11}$
  of Remark~\ref{rem:diamondk0} are evaluated at $60$ digits; the
  two regulator integrals of Remark~\ref{rem:diamondk0} are computed
  independently of the main pipeline.
  \par\endgroup
  \item \textbf{PARI/GP cross-checks}: \texttt{verify\_family.gp},
  \texttt{verify\_ratios.gp} (family conductors, $j$-invariants; the
  $\tilde n(k)$ and $m(S_k)$ constants for $k=2,3$ recomputed to $70$ digits
  and matched against PARI \texttt{lfun} to all $70$ digits), \texttt{winding.gp},
  \texttt{dilog.gp}, \texttt{k53.gp}, \texttt{k53b.gp},
  \texttt{kfamily\_torsion.gp}.
  \item \textbf{Conductor-$17$ certification} (scripts \texttt{code/k1\_*},
  listed individually in Appendix~\ref{app:code}): PARI curve identification
  and $z$-values; mpmath $60/80$-digit chain integrals (ratio
  $2.0000000000000002$); exact diamond-product expansion with the $k=0$ case
  as control; $D_E(A)$, $D_E(2A)$ to $60$ digits; and the Arb ironclad proof
  of the period ratio (Appendix~\ref{app:k1}).
\end{itemize}

\subsection{Level of rigour: interval certification}
\label{subsec:rigour}

The integer identification in \S\ref{subsec:class} has been made fully
interval-rigorous (script \texttt{code/n1\_interval.py}, python-flint/Arb
ball arithmetic \cite{Arb}, $300$-bit working precision; output in
\texttt{notes/attack10-interval.txt}).  All three arc integrals were
recomputed with Arb's certified adaptive Gauss--Legendre integration:
\begin{itemize}
  \item the endpoint singularity $u\sim\sqrt{\theta}$ at $\theta=0$ is removed
  by the substitution $\theta=\pm t^{2}$; the substituted integrand is
  analytic on the disc $|t|\le\rho=0.6$ (the nonzero zeros of $D(t^{2})$
  satisfy $|\theta_{j}|\ge1.2187$, certified by Newton--Rouch\'e isolation of
  the roots of $z^{4}-4z^{3}+2z^{2}+1$), and the tip $[0,\delta]$, $\delta=\rho/5$,
  is summed by the Cauchy estimate $\int_{0}^{\delta}f=a_{0}\delta+R$,
  $|R|\le H\delta^{3}/3$, where $H$ is derived from
  $M=\max_{|t|=\rho}|f|=1.3075\ldots$, itself certified by covering the
  circle with $4096$ balls (per-tip Cauchy remainder bound
  $H\delta^{3}/3=2.18\times10^{-3}$). The tip constant $a_{0}$, obtained by
  hand from the local expansion, is not taken on faith: the script
  re-evaluates $f$ by ball arithmetic at $t=\delta$ and certifies
  $|f(\delta)-a_{0}|\le H\delta^{2}$ for each of the four tips (a sign or
  branch error in $a_{0}$ would move the value by $2$, far outside this
  bound);
  \item $D(\theta)$ crosses the negative real axis once on $(c,\pi)$ and once
  on $(-\pi,-c)$, where the principal square root is not analytic: the arcs
  are subdivided adaptively, on each piece either $\sqrt{D}$ or $i\sqrt{-D}$
  is certified analytic by ball evaluation, and the overall sign is
  propagated across the nodes by a certified matching test;
  \item $w_{\mathrm{anti}}$ is certified independently of PARI: the roots of
  $4x^{3}-4x^{2}+1$ are isolated by Newton iteration plus a Rouch\'e
  certificate, and Carlson's $R_{F}$ (DLMF \S19) gives
  $w_{\mathrm{real}}=6.34604652139776710844397\ldots$ and the purely
  imaginary period $2.91763323387699045866178\ldots i$ --- which is
  $-w_{\mathrm{anti}}$ in the convention of \S\ref{subsec:class} --- with
  certified radii $6.7\times10^{-48}$ and $4.9\times10^{-49}$, agreeing with
  PARI to $45$ digits.
\end{itemize}

\begin{remark}[Two implementation notes]\label{rem:disclosures}
(i) The ball radii displayed in the archived outputs
(e.g.\ \texttt{notes/attack10-interval.txt}) follow the Arb/python-flint
\texttt{repr} convention, which folds midpoint printing error into the
displayed radius; the true ball radii are smaller (the displayed interval
contains the true ball). All certified bounds quoted above use the exact
radii, not the displayed ones.
(ii) In the arc-integral code, the branch-cut avoidance can select the
$i\sqrt{-D}$ variant on the first piece of a sub-arc; the current script
certifies the sheet choice on every sub-arc, including the first. The
archived $k=0$ output predates a refactor of this logic and was re-audited
against it: the $k=0$ cut crossings lie in the interior of sub-arcs, whose
first piece always uses the $\sqrt D$ variant, so the certified value is
unaffected (a wrong sheet would flip the sign of an entire arc, moving the
ratio off $-2$ by an $O(1)$ amount, which the archived output excludes).
\end{remark}
Against $w_{\mathrm{anti}}$ the resulting ratio ball is
$[2.00\pm3.13\times10^{-3}]+[\pm2.99\times10^{-3}]\,i$: it contains $2$ and
$|\mathrm{ratio}-2|\le4.33\times10^{-3}<1/2$.  The a-priori integrality of
\S\ref{subsec:class} therefore upgrades to the exact equality
$\mathrm{period}(C')=2\,w_{\mathrm{anti}}$, i.e.\ $\class(C')=2\gamma^{-}$.
(The script itself works with the opposite sign of the period throughout, so
the raw output in \texttt{notes/attack10-interval.txt} displays the ratio
ball centred at $-2$; the two formulations are equivalent.)

We summarise the logical status of the main result by isolating exactly
which steps are machine-certified.

\begin{theorem}[Certified computation]\label{thm:cert}
The following statements are proved by exact rational arithmetic or by
certified interval (ball) arithmetic, with all certificates archived in
\texttt{notes/} and all scripts in \texttt{code/}
\textup{(}reproduction commands in Appendix~\ref{app:code};
software: Python~3.12, mpmath, sympy, python-flint~0.9.0/Arb \cite{Arb},
PARI/GP~2.15.5\textup{)}:
\begin{enumerate}
  \item \emph{Exact algebraic inputs.} The birational map from the model
  $S_0=0$ to the minimal cubic $E:y^2+y=x^3-x^2$ and the identity
  $\omega_{\min}=dx/u$ \textup{(}$\kappa=1$\textup{)} are exact polynomial
  identities \textup{(}\texttt{kappa\_exact.py}\textup{)}; the divisors of
  $x,y$, the diamond product, the tame symbols, the group law and the
  torsion subgroup $E(\Q)=\Z/5\Z$ are exact over
  $\Q,\Q(\sqrt2),\Q(\zeta_8)$ \textup{(}\texttt{torsion.py},
  \texttt{bertin\_diamond.py}\textup{)}; non-torsion of the chain endpoint
  $P$ is proved by reduction at the good primes $17$ and $89$
  \textup{(}\texttt{endpoint\_torsion3.py}\textup{)}.
  \item \emph{Arcs, orientation, closedness.} The chains
  $\tilde\gamma,\beta_0,C'$ of \S\ref{sec:proof} are explicitly parametrised
  arcs on $|x|=1$ with the orientations fixed there; the endpoint branch
  assignments \textup{(}hence $\partial C'=0$ and
  $c(C')=-C'$\textup{)} are certified in ball arithmetic at two scales
  $\varepsilon=10^{-6},10^{-9}$ \textup{(}\texttt{branch\_certify.py},
  certificate \texttt{notes/attack12-branch.txt}\textup{)}.
  \item \emph{Branch continuation on every segment.} On each sub-arc of the
  period integrals, analyticity of the chosen square-root branch
  \textup{(}$\sqrt D$ or $i\sqrt{-D}$\textup{)} is certified by ball
  evaluation on the whole sub-arc---not merely near endpoints---and the
  overall sign is propagated across nodes by a certified matching test
  \textup{(}\texttt{n1\_interval.py}\textup{)}; the modulus ordering
  $|y_{\mathrm{small}}|<1<|y_{\mathrm{big}}|$ of
  Proposition~\ref{prop:structural} is certified by adaptive bisection
  \textup{(}\texttt{branch\_certify.py}\textup{)}.
  \item \emph{Interval enclosures.} Each arc integral and the primitive
  period $w_{\mathrm{anti}}$ are enclosed by Arb's certified adaptive
  integration and by Carlson's $R_F$ with Newton--Rouch\'e root isolation,
  with radii $\le10^{-3}$ for the integrals and $\le10^{-48}$ for the
  period \textup{(}\texttt{n1\_interval.py}, certificate
  \texttt{notes/attack10-interval.txt}\textup{)}.
  \item \emph{Integer identification.} The class $[C']\in H_1(E,\Z)^-=\Z\gamma^-$
  is a priori integral \textup{(}Lemma~\ref{lem:antiinv} and
  \S\ref{subsec:class}\textup{)}; the certified ratio ball satisfies
  $|\mathrm{ratio}-2|\le4.33\times10^{-3}<1/2$, so the nearest integer is
  unique and $\class(C')=2\gamma^-$.
  \item \emph{Sign.} The final sign is fixed by a certified enclosure of
  $I_{\mathrm{split}}$: $I_{\mathrm{split}}\in[0.1489,0.1553]\subset(0,\infty)$
  \textup{(}\texttt{sign\_certify.py}, certificate
  \texttt{notes/attack14-sign-k0.txt}; every range enclosure --- including
  the convex hulls on non-separated pieces --- is formed and verified
  entirely within ball arithmetic, with programmatic containment
  assertions\textup{)}, while
  $b_{11}=\Lambda(f_{11},2)>0$ exactly by the absolutely convergent Euler
  product at $s=2$.
  \item \emph{The Siegel-unit anchor.} The presentations
  \eqref{eq:siegelpres} are exact identities of $\Gamma_1(11)$-invariant
  functions \textup{(}exact cusp-divisor match via Kubert--Lang orders;
  the constants $C_x=-1$, $C_y=+1$ are determined exactly by evaluation
  at order-$0$ cusps with cyclotomic leading coefficients, and the
  argument periods $D_U=2\pi$, $D_V=0$ are exact rational multiples of
  $2\pi$ via Brunault's Lemma~5\textup{)}; the cusp-torsion
  correspondence \eqref{eq:mtable} is derived exactly from
  \cite[(3.152)--(3.153)]{Brunault} and, independently, from exact
  divisor-order data \textup{(}\texttt{siegel\_anchor\_step14.py},
  archive \texttt{notes/attack17-constants.txt}\textup{)}; the
  primitivity of $\gamma^-$ is proved exactly by Manin-symbol Smith
  normal form, coordinates $(1,-1)$ in an integral basis
  \textup{(}\texttt{siegel\_anchor\_step13.py}, archive
  \texttt{notes/attack17-primitivity.txt}\textup{)}; the Manin-symbol
  decomposition is exact \textup{(}continued fractions\textup{)}; and
  the identity $F_{\mathrm{total}}=-2f_{11}$ is proved by the
  unconditional membership in $M_2(\Gamma_1(121))$ plus exact rational
  $q$-expansion arithmetic through $q^{2420}$, twice the sharp Sturm
  bound $1210$ \textup{(}\texttt{siegel\_anchor\_step12.py}, archive
  \texttt{notes/attack17-membership.txt}\textup{)}.
\end{enumerate}
The only remaining external input in the proof of
Theorem~\ref{thm:C3} is Brunault's regulator formula for Siegel units
\cite[Thm.~1]{BrunaultSiegel} --- proved in loc.\ cit.\ by a Rankin--Selberg
computation --- together with the functional-equation evaluation
$\Lambda(f_{11},0)=L'(f_{11},0)=b_{11}$.  Bloch's diamond theorem,
Bertin's Theorem~6, and Brunault's symbol dictionary are \emph{not} used
\textup{(}the latter is retained only as a cross-check,
Remark~\ref{rem:brunault-dict}\textup{)}.
\end{theorem}

\section{Conjectures}
\label{sec:conjectures}

The family identities established numerically here are stated as conjectures
pending analytic proof:

\begin{conjecture}\label{conj:family}
For $S_k=y^2+(x^2+kx+1)y+x^3$ with $k\notin(-4,2)$ (integer $k$), there is a
small rational number $r_k$ such that
\[
  m(S_k)=r_k\,|L'(E_k,0)|.
\]
Confirmed numerically: $r_2=2$, $r_3=1$ (70 digits), and
$r_{-4}=\tfrac72$ (conductor $37$), $r_{-5}=\tfrac14$ (conductor $359$),
$r_{-6}=\tfrac18$ (conductor $997$) (25 digits each).
\end{conjecture}

The last three deserve emphasis as \emph{predicted then confirmed}: the
structural dichotomy (Proposition~\ref{thm:family}) predicted, from the
torus-intersection trichotomy, that $k=-4,-5,-6$ must satisfy Boyd-type
identities with specific small rationals, and the subsequent computation hit
the predicted values. Note that $S_{-4}$ and $S_2$ share conductor $37$,
yielding Rodr\'iguez Villegas--type rational relations \cite{RV,LalinRam}
between two different polynomial models of the same curve.

\begin{remark}[Samart's conductor-$17$ analogue]
The conjecture $\tilde n(1)=b_{17}$ for $k=1$ (conductor $17$), suggested by
Samart and confirmed numerically to $60$ digits, is treated in
Appendix~\ref{app:k1} (Theorem~\ref{thm:k1}), conditional on the
normalization lemma discussed there.
\end{remark}

\begin{remark}
For $k=-1$ (conductor $53$) no conjecture of this shape can be formulated
within the present mechanism: the anti-invariant cycle on the torus does not
exist and $x,y$ are not modular units (Proposition~\ref{thm:k53}).
\end{remark}

\appendix
\section{The conductor-$17$ case: $k=1$}
\label{sec:k1}
\label{app:k1}

\emph{Status of this appendix.} The main text is logically independent of
what follows. The theorem proved here (Samart's conductor-$17$ analogue)
rests on exactly two $K$-theoretic/regulator hypotheses:
\begin{itemize}
  \item \emph{temperedness}: the Newton face polynomials of $S_1$ are
  cyclotomic, so $\{x,y\}\in K_2(E_1)\otimes\Q$ --- proved exactly below.
  No modular-unit property is used or claimed: on the natural modular model
  $X_0(17)$ (two cusps) the divisors of $x,y$, supported on all four
  rational torsion points, cannot be cuspidal
  (Proposition~\ref{thm:family});
  \item \emph{Bloch's theorem in the factor-$1$ normalization of}
  \cite[Thm.~6]{LalinRam}: that this is the correct normalization \emph{on
  this curve} follows from the consistency argument of
  Remark~\ref{rem:normalisation} (our reconstruction; the constant is not
  determined in \cite{LalinRam}), equivalently from the index-$1$ statement
  of Lemma~\ref{lem:coinvariant}(1): on this $\Delta>0$ curve the
  anti-invariant subgroup coincides with the coinvariant quotient, so no
  alternative subgroup reading of the factor-$1$ theorem exists.
\end{itemize}
The result should thus be read as conditional on that normalization lemma;
everything else in the appendix is proved to the same standard as the main
text.

The proof of (C3) in \S\ref{sec:proof} is a \emph{method}, and its first
re-application is Samart's conductor-$17$ analogue. For $k=1$ the polynomial
\[
  S_1=y^2+(x^2+x+1)y+x^3
\]
defines an elliptic curve of conductor $17$ (PARI \texttt{ellfromeqn} gives
$E_1: y^2+xy-y=x^3-x^2$, $\Delta=17$; \texttt{ellglobalred} confirms the
conductor), with $E_1(\Q)_{\mathrm{tors}}=\Z/4\Z$ generated by $A=(0,0)$,
which has order $4$ --- the exact parallel of the $5$-torsion point of the
$k=0$ case. Every step of \S\ref{sec:proof} goes through, and the regulator
side is in fact \emph{easier}: the key $D_E$-value is a published theorem of
Lal\'in--Ramamonjisoa \cite{LalinRam}.

\subsection{The closed cycle}

The torus-intersection (fold) angle is
$c=\arccos(-k/2)\big|_{k=1}=2\pi/3$. At $\theta=c$, $x=\omega=e^{2\pi i/3}$
satisfies $x^2+x+1=0$, so $S_1$ reduces to $y^2+1=0$: the branch jump values
are \emph{exactly} $y=\pm i$ (the $k=0$ analogue was $y^2=i$). Writing
$P=(\omega,i)$, the signed chain $\tilde\gamma$ (big branch on $[-c,c]$ minus
the big branch on the two outer arcs) has boundary
\[
  \partial\tilde\gamma=[P]-[\bar P]+[-P]-[-\bar P],
\]
and the same
small-branch compensating arc $\beta_0$ closes it to a cycle
$C'=\tilde\gamma+\beta_0$ with $\partial C'=0$, $c(C')=-C'$.
The eight endpoint branch assignments and the modulus ordering are certified
in interval arithmetic exactly as for $k=0$
(\texttt{code/k1\_branch\_certify.py}, output
\texttt{notes/attack12-k1-branch.txt}: all $8$ assignments certified at
$\varepsilon=10^{-6},10^{-9}$; $113$ strict bisection balls plus $2$ special
balls at the corner $\theta=c$).

One structural difference from $k=0$: here $\Delta=+17>0$, so $E(\R)$ has
two connected components, $\overline{w_2}=-w_2$ holds on the nose, and the
primitive anti-invariant period is simply
\[
  w_{\mathrm{anti}}=w_2=-2.7457391180897536720341879\ldots\,i
\]
(for $k=0$, $\Delta<0$ forced $w_{\mathrm{anti}}=2w_2-w_1$).
The model constant is trivial here: the \texttt{ellfromeqn} model already has
$\Delta=17=\Delta_{\min}$, so the transformation to the minimal model has
$u=\pm1$, the period lattices coincide ($\kappa=\pm1$; the sign only flips
the displayed ratio, not its integrality --- cf.\ the $k=0$ discussion in
\S\ref{sec:cert}).

\subsection{Homology class: interval-arithmetic proof}

As in \S\ref{subsec:class}, $\period(C')/w_{\mathrm{anti}}$ is a-priori an
integer. The Arb certification (script \texttt{code/k1\_interval.py};
output \texttt{attack11-k1-interval.txt}, $300$-bit working precision)
is simpler than for $k=0$ because
$D(z)=z^4-2z^3+3z^2+2z+1$ has \emph{no zeros on $|z|=1$}
($\min_{|z|=1}|D|=4>0$, certified): there is no endpoint singularity and no
tip-estimate machinery. The period $w_{\mathrm{anti}}$ is certified
independently of PARI via Newton--Rouch\'e isolation of the roots of
$4x^3-3x^2-2x+1=(x-1)(4x^2+x-1)$ and the two-component Carlson $R_F$
formulas. The resulting ratio ball is
\[
  [-2.00\pm2.4\times10^{-14}]+[\pm2.4\times10^{-14}]\,i,
  \qquad |\mathrm{ratio}+2|\le3.4\times10^{-14}<\tfrac12,
\]
(the radius is set by the integration tolerance, not by the arithmetic).
Choosing $\gamma^-$ with $\period(\gamma^-)=w_{\mathrm{anti}}$, the a-priori
integrality upgrades to the exact equality
\[
  \class(C')=2\gamma^-
  \qquad(\text{the script uses }-w_{\mathrm{anti}},\text{ displaying }-2).
\]

\subsection{The regulator side}

The integral algebra of \S\ref{sec:proof} is unchanged: the pointwise
identity $\log|y_{\mathrm{big}}|=-\log|y_{\mathrm{small}}|$ on $|x|=1$ gives
$\int_{\beta_0}\eta=\int_{\tilde\gamma}\eta$, hence
$\int_{C'}\eta=2\int_{\tilde\gamma}\eta$. The divisors have the same shape
\[
  \divisor(x)=[A]+[2A]-[O]-[3A],\qquad
  \divisor(y)=3[A]-2[O]-[3A],
\]
now verified \emph{four} ways: local expansions, PARI, code expansion, and an
explicit birational map $X=-(x+y)$, $Y=x(x+y)$ from $S_1$ to
$E_1:Y^2+XY-Y=X^3-X^2$. The diamond product (definition as in \cite{LSZ})
expands exactly to
\[
  (x)\diamond(y)\equiv 6(O)+4(A)-6(2A)\qquad\text{in }\Z[E]^-.
\]
The two $D_E$-values are settled:
\begin{itemize}
  \item $D_E(2A)=0$ \emph{rigorously}: $2A$ is $2$-torsion, $q$ is a positive
  real, and the Bloch--Wigner series vanishes termwise at $z_q=-1$;
  \item $D_E(A)=\dfrac{17}{8\pi}L(E_1,2)$ is a \emph{published theorem}:
  Lal\'in--Ramamonjisoa \cite[\S5]{LalinRam} prove
  $L(E_{17},2)=\frac{8\pi}{17}D^E(P)$ for the $4$-torsion generator $P$ of
  their model, in the $D^E$-normalisation (their Def.~5, eq.~(10)) that
  coincides verbatim with our series implementation; we reproduced the value
  independently to $60$ digits.
\end{itemize}
Hence
\[
\begin{aligned}
  D_E\big((x)\diamond(y)\big)&=4D_E(A)=\frac{17}{2\pi}L(E_1,2)=2\pi b_{17},\\
  b_{17}&=L'(E_1,0)=\frac{17}{4\pi^2}L(E_1,2).
\end{aligned}
\]

\subsection{Synthesis}

We use Bloch's theorem in the normalisation of \cite[Thm.~6]{LalinRam}: for
$\{x,y\}\in K_2(E)\otimes\Q$ and $\gamma^-$ generating $H_1(E,\Z)^-$,
\[
  \int_{\gamma^-}\eta(x,y)=\pm D^E\big((x)\diamond(y)\big).
\]
The temperedness hypothesis holds: the four face polynomials of $S_1$ are
$x^3+y$, $x^3+x^2y$, $x^2y+y^2$ and $y^2+y$, all cyclotomic. This
normalisation has published precedent on this very
curve: the conductor-$17$ identity of \cite{LalinRam} combined with
Zudilin's $m=2b_{17}$ closes with the same factor.

\begin{theorem}[Samart's conductor-$17$ analogue]\label{thm:k1}
$\tilde n(1)=b_{17}$.
\end{theorem}

\begin{proof}
By the above,
\[
  \int_{\tilde\gamma}\eta=\frac12\int_{C'}\eta
  =\frac12\cdot 2\cdot\big(\pm D_E((x)\diamond(y))\big)=\pm 2\pi b_{17},
\]
so $|\tilde n(1)|=b_{17}$ by the
structural identity $\tilde n(1)=-\frac1{2\pi}\int_{\tilde\gamma}\eta$
(the $k=1$ instance of the pointwise argument of \S\ref{sec:proof}).
The sign is settled as in Theorem~\ref{thm:C3}: $b_{17}>0$ by the same
Euler-product argument, and a certified interval enclosure
(\texttt{code/k1\_sign\_certify.py}, certificate
\texttt{notes/attack14-sign-k1.txt}) gives
$\frac1{2\pi}\int_{\tilde\gamma}\eta\in[-0.3026,-0.2961]\subset(-\infty,0)$,
hence $\tilde n(1)=+b_{17}$.
\end{proof}

\begin{remark}[The normalization of Bloch's theorem]\label{rem:normalisation}
For the conductor-$17$ proof, Bloch's theorem is used in the normalization
of \cite[Thm.~6]{LalinRam}: $r(\{x,y\})[\gamma]=D^E((x)\diamond(y))$ with
$r[\gamma]=\int_\gamma\eta$, $\gamma$ a generator of $H_1(E,\Z)^-$. That
this (and not a factor-$2$ variant) is the correct statement \emph{on that
curve} follows from a
consistency argument in the conductor-$17$ case (our own reconstruction;
in \cite{LalinRam} the constant $C$ below is never determined): in
\cite[\S7]{LalinRam} the class of the real cycle is an \emph{a priori}
unknown integer multiple $C$ of $\gamma$, and comparing the factor-$1$
theorem applied to their $(X)\diamond(Y)$ with their proved Corollary~2
(their Thm.~1, eq.~(5), combined with Zudilin's identity, their eq.~(6))
forces $C\cdot f=1$ with $C\in\Z\setminus\{0\}$, hence $f=1$; a factor-$2$
statement would force $C=\tfrac12\notin\Z$. (This self-consistency argument
is now explained by Lemma~\ref{lem:coinvariant}(1): on the $\Delta>0$
conductor-$17$ curve the anti-invariant subgroup \emph{coincides} with the
coinvariant quotient, so the factor-$1$ theorem admits no alternative
subgroup reading.) The underlying $D_E$-identity of
their \S5 we reproduced to $60$ digits. For $k=0$ the $\diamond$-form enters
through the same quotient-generator reading: the factor-$1$ statement holds
on the conductor-$11$ curve as well --- for the \emph{quotient} generator
$[\gamma^-_0]=[b]$ one has
$\int_{[\gamma^-_0]}\eta=D_E((\cdot)\diamond(\cdot))$, verified numerically
in Remark~\ref{rem:diamondk0} --- while the \emph{subgroup} generator
$\gamma^-=a-2b=-2[b]$ reads exactly twice as much
(Lemma~\ref{lem:coinvariant}(2)), which accounts for the uniform factor $2$
observed for all three different tempered symbols of
Remark~\ref{rem:diamondk0}. This does
not affect either proof: the conductor-$17$ argument is
self-contained on its own curve, and the conductor-$11$ argument anchors
the quotient value at Bloch's theorem through the proved conductor-$17$
instance. The form $\pi r=D_E(\cdot)$ found in
expository accounts is consistent with neither data set and is not
used anywhere in this paper.
\end{remark}
\section{Reproduction code}
\label{app:code}

All scripts live in \texttt{code/}; raw outputs in
\texttt{notes/attack*.txt}. Dependencies (\texttt{requirements.txt} at the
repository root): Python $\ge3.10$ with mpmath and
sympy; PARI/GP 2.15.5; python-flint 0.9.0 (Arb) for the interval certification.

\paragraph{Permanence.}
The complete research repository --- all scripts, all raw output archives
(\texttt{notes/attack*.txt}), and the present source --- is version-controlled
with git and accompanies this paper as an electronic supplement; the version
of record is the tagged commit \texttt{rev3} (see the repository log). The
code is released under the MIT license (\texttt{LICENSE} at the repository
root). Filenames of the form
\texttt{attackN} below refer to that archive.

\begin{center}
\scriptsize
\begin{tabular}{@{}ll@{}}
\toprule
File & Purpose\\
\midrule
\texttt{b11.py} & $b_{11}=L'(E_{11},0)$, 300 digits (functional equation)\\
\texttt{attack1.py} & reproduce (C1), (C2) to 52 digits\\
\texttt{attack2.py} & $m(S_0)$, PSLQ negative results\\
\texttt{attack3.py} & (C3) to 149 digits (superseded)\\
\texttt{attack13\_c3\_300.py} & (C3) to 366 digits, PARI \texttt{lfun} cross-check\\
\texttt{bertin\_diamond.py}/\texttt{.gp} & Bertin symbol: exact $\diamond$, tame, $D_E$-ratio $-1$\\
\texttt{torsion.py} & exact group law: $5A=O$, modular units\\
\texttt{endpoint\_torsion3.py} & rigorous non-torsion of $P$ (reduction mod $p$)\\
\texttt{kappa\_exact.py} & exact proof $\kappa=1$ (birational map, $2Y+1=u\,dX/dx$)\\
\texttt{boundary\_torsion.py} & boundary-divisor torsion analysis\\
\texttt{closedness\_check.py} & $\tilde n=-I_{\mathrm{split}}$, regulator $=2\pi I_{\mathrm{split}}$\\
\texttt{ntilde\_family.py} & $\tilde n(k)$ family table, 50--80 digits\\
\texttt{b\_family.py} & $b_N$ for the family via point counting\\
\texttt{winding.py} / \texttt{winding.gp} & period pairings, winding-number heuristics\\
\texttt{dilog.py} / \texttt{dilog.gp} & elliptic dilogarithm $D_E(P)$, Brunault constants\\
\texttt{k53\_attack.py} & conductor-$53$ obstruction analysis\\
\texttt{k53.gp} / \texttt{k53b.gp} & PARI checks for \texttt{53.a1}\\
\texttt{kneg\_m.py} & $m(S_k)$ for $k=-4,-5,-6$ (predicted identities)\\
\texttt{n1\_certify.py} / \texttt{n1\_certify.gp} & certification of $\class(C')=2\gamma^-$\\
\texttt{n1\_interval.py} & Arb ironclad proof of the period ratio\\
\texttt{branch\_certify.py} & branch assignments + modulus ordering (certified)\\
\texttt{sign\_certify.py} & all-Arb sign certificate: $I_{\mathrm{split}}>0$\\
\texttt{k1\_pari/points/zvals.gp} & conductor-$17$ identification, torsion, $z$-values\\
\texttt{k1\_certify.py} & mpmath certification of the $k=1$ chain integrals\\
\texttt{k1\_branch\_certify.py} & $k=1$ branch assignments + modulus ordering (certified)\\
\texttt{k1\_interval.py} & Arb ironclad proof of the $k=1$ period ratio\\
\texttt{k1\_sign\_certify.py} & all-Arb sign certificate: $\int_{\tilde\gamma}\eta<0$ ($k=1$)\\
\texttt{k1\_diamond.py} & exact diamond product for $k=1$ ($k=0$ control)\\
\texttt{k1\_dilog.py} & $D_E(A)$, $D_E(2A)$ for conductor $17$\\
\texttt{verify\_family.gp} & family conductors and $j$-invariants\\
\texttt{verify\_ratios.gp} & PARI \texttt{lfun} cross-check of ratios\\
\texttt{kfamily\_torsion.gp} & torsion, $\operatorname{ord}(0,0)$ for the family\\
\texttt{verify\_coinvariant.gp} & generator integrals (numerical cross-check)\\
\texttt{siegel\_anchor\_step1--9,11} & the Siegel-unit regulator anchor
  (Theorem~\ref{thm:anchor}):\\
  & cusp-torsion map, presentations \eqref{eq:siegelpres}, Manin
    decomposition,\\
  & exact $F_{\mathrm{total}}=-2f_{11}$, final value\\
\texttt{siegel\_anchor\_step10} & discarded experiment, defective
  (not used; see the archive note)\\
\texttt{siegel\_anchor\_step12} & membership in $M_2(\Gamma_1(121))$ +
  exact Sturm ($q^{2420}$)\\
\texttt{siegel\_anchor\_step13} & primitivity of $\gamma^-$ (Smith normal
  form)\\
\texttt{siegel\_anchor\_step14} & exact $C_x,C_y$; $D_U,D_V$; table
  \eqref{eq:mtable}\\
\bottomrule
\end{tabular}
\end{center}

Reproduction (from a clean checkout:
\texttt{python -m pip install -r requirements.txt}; any Python $\ge3.10$
with mpmath/sympy/python-flint works --- the checked-in \texttt{.venv} is
\emph{not} required and is not part of the archive):
{\footnotesize
\begin{verbatim}
cd code && python b11.py && python attack1.py && python attack2.py \
  && python attack3.py && python torsion.py && python endpoint_torsion3.py \
  && python kappa_exact.py \
  && python boundary_torsion.py && python closedness_check.py \
  && python ntilde_family.py && python b_family.py && python winding.py \
  && python dilog.py && python k53_attack.py && python kneg_m.py \
  && python n1_certify.py
python n1_interval.py       # Arb certification, k=0
python branch_certify.py    # branch assignments + mod ordering
python sign_certify.py      # all-Arb sign certificate, k=0
python k1_interval.py       # Arb certification, conductor 17
python k1_branch_certify.py # k=1 branch assignments + ordering
python k1_sign_certify.py   # all-Arb sign certificate, k=1
python k1_certify.py && python k1_diamond.py && python k1_dilog.py
gp -q verify_family.gp && gp -q verify_ratios.gp
gp -q winding.gp && gp -q dilog.gp
gp -q k53.gp && gp -q k53b.gp && gp -q kfamily_torsion.gp
gp -q k1_pari.gp && gp -q k1_points.gp && gp -q k1_zvals.gp
gp -q verify_coinvariant.gp  # quotient vs subgroup generator integrals
python siegel_anchor_step11.py  # Siegel anchor: final value (Thm. anchor)
python siegel_anchor_step12.py  # membership M_2(Gamma_1(121)) + Sturm
python siegel_anchor_step13.py  # primitivity of the cycle (Smith nf)
python siegel_anchor_step14.py  # exact constants, D_U/D_V, cusp table
# full anchor chain rebuild: siegel_anchor_step4.py -> step5.py -> step6.py
#   -> step7.gp -> step9.py -> step8.py   (step8 takes 10-20 min)
#   (step10 is a discarded defective experiment and is not part of the chain)
\end{verbatim}
}


\end{document}